\documentclass{amsart}

\usepackage[latin1]{inputenc}
\usepackage[T1]{fontenc}
\usepackage[english, french]{babel}
\usepackage{graphicx}
\usepackage{endnotes}
\usepackage{amsmath, amsthm}
\usepackage{amssymb}
\usepackage{xcolor}
\usepackage{graphics,graphicx}
\usepackage{tikz}
\usetikzlibrary{trees}
\usepackage{hyperref}
\hypersetup{pdfpagemode=UseOutlines, pdfstartview=Fit,    
pdffitwindow=true, pdfpagelayout=SinglePage, pdftoolbar=true, pdfmenubar=true, bookmarksopen=false, bookmarksnumbered=true, colorlinks=true, linkcolor=blue, urlcolor=blue,               
anchorcolor=black, citecolor=blue, frenchlinks=true, pdfborder={0 0 0}}

\numberwithin{equation}{section} 
\numberwithin{figure}{section} 
\ifx\thesection\undefined
\newtheorem{thm}{Theorem}
\else
\newtheorem{thm}{Theorem}[section]
\fi
  \theoremstyle{definition}
  \newtheorem{defi}[thm]{Definition}
  \theoremstyle{definition}
  \newtheorem{exple}[thm]{Example}
   \theoremstyle{definition}
  \newtheorem{prop}[thm]{Proposition}
   \theoremstyle{definition}
  
  \theoremstyle{definition}
  \newtheorem{lem}[thm]{Lemma}
    \theoremstyle{definition}
  \newtheorem{cor}[thm]{Corollary}
  \theoremstyle{theorem}

\newcommand{\Sk}{  \mathcal{C}_k  }
\newcommand{\Spk}{  \mathcal{C}^p_k  }
\newcommand{\Sak}{  \mathcal{C}^a_k  }
\newcommand{\Spak}{  \mathcal{C}^{pa}_k  }
\newcommand{\Sck}{  \mathcal{C}^{c}_k  }
\newcommand{\Scuck}{  \mathcal{C}^{cm}_k  }

\newcommand{\Skmu}{  \mathcal{C}_k  }

\newcommand{\Spkmu}{  \mathcal{C}^p_{k-1}  }

\newcommand{\So}{  \mathcal{C}_0  }
\newcommand{\Spo}{  \mathcal{C}^p_0  }
\newcommand{\Sao}{  \mathcal{C}^a_0  }
\newcommand{\Spao}{  \mathcal{C}^{pa}_0  }

\newcommand{\Smu}{  \mathcal{C}_{-1}  }
\newcommand{\Spmu}{  \mathcal{C}^p_{-1}  }

\newcommand{\Zk}{  \mathbf{{C}}_k  }
\newcommand{\Zpk}{  \mathbf{{C}}^p_k  }
\newcommand{\Zak}{  \mathbf{{C}}^a_k  }
\newcommand{\Zpak}{  \mathbf{{C}}^{pa}_k  }

\newcommand{\Zo}{  \mathbf{{C}}_0  }
\newcommand{\Zpo}{  \mathbf{{C}}^p_0  }
\newcommand{\Zao}{  \mathbf{{C}}^a_0  }
\newcommand{\Zpao}{  \mathbf{{C}}^{pa}_0  }

\newcommand{\Zmu}{  \mathbf{{C}}_{-1}  }
\newcommand{\Zpmu}{  \mathbf{{C}}^p_{-1}  }

\newcommand{\Zkmu}{  \mathbf{{C}}_{k-1}  }
\newcommand{\Zpkmu}{  \mathbf{{C}}^p_{k-1}  }

\newcommand{\Zkt}{  \mathbf{{C}}_{{k,t}}  }
\newcommand{\Zpkt}{  \mathbf{{C}}^p_{{k,t}}  }
\newcommand{\Zakt}{  \mathbf{{C}}^a_{{k,t}}  }
\newcommand{\Zpakt}{  \mathbf{{C}}^{pa}_{{k,t}}  }

\newcommand{\ZAkt}{  \mathbf{{C}}^A_{{k,t}}  }
\newcommand{\ZPAkt}{  \mathbf{{C}}^{pA}_{{k,t}}  }

\newcommand{\Zot}{  \mathbf{{C}}_{{0,t}}  }
\newcommand{\Zpot}{  \mathbf{{C}}^p_{{0,t}}  }

\newcommand{\ZAot}{  \mathbf{{C}}^A_{{0,t}}  }
\newcommand{\ZPAot}{  \mathbf{{C}}^{pA}_{{0,t}}  }

\newcommand{\Zmut}{  \mathbf{{C}}_{{-1,t}}  }
\newcommand{\Zpmut}{  \mathbf{{C}}^p_{{-1,t}}  }

\newcommand{\ZAmut}{  \mathbf{{C}}^A_{{-1,t}}  }
\newcommand{\ZPAmut}{  \mathbf{{C}}^{pA}_{{-1,t}}  }

\newcommand{\Zaut}{  \mathbf{{C}}^a_{{1,t}}  }
\newcommand{\Zpaut}{  \mathbf{{C}}^{pa}_{{1,t}}  }
\newcommand{\ZAut}{  \mathbf{{C}}^A_{{1,t}}  }
\newcommand{\ZPAut}{  \mathbf{{C}}^{pA}_{{1,t}}  }

\newcommand{\Zkmut}{  \mathbf{{C}}_{{k-1},t}  }
\newcommand{\Zpkmut}{  \mathbf{{C}}^p_{{k-1},t}  }
\newcommand{\ZAkmut}{  \mathbf{{C}}^A_{{k-1,t }}  }
\newcommand{\ZPAkmut}{  \mathbf{{C}}^{pA}_{{k-1,t }}  }

\newcommand{\SW}{{\Sigma W}}
\newcommand{\SWt}{{\Sigma W}_t}

\newcommand{\p}{\mathbf{p}}

\newcommand{\comm}{\mathbf{C}_{\operatorname{Comm}}}
\newcommand{\perm}{\mathbf{C}_{\operatorname{Perm}}}
\newcommand{\spl}{\Sigma \mathbf{C}_{\operatorname{PreLie}}}

\newcommand{\Slie}{\Sigma \mathbf{C}_{\operatorname{Lie}}}
\newcommand{\Stlie}{\Sigma_t \mathbf{C}_{\operatorname{Lie}}}

\newcommand{\HAL}{\operatorname{HAL}}
\newcommand{\HALp}{\operatorname{HAL^p}}
\newcommand{\HALa}{\operatorname{HAL^A}}
\newcommand{\HALpa}{\operatorname{HAL^{pA}}}

\title{Action of the symmetric groups on the homology of the hypertree posets}
\author{Bérénice Oger}
\thanks{Institut Camille Jordan, UMR 5208 \\
 Université Claude Bernard Lyon 1\\
 Bât. Jean Braconnier n°101 \\
 43 Bd du 11 novembre 1918 \\
 69622 Villeurbanne Cedex \\ 
 \email{oger@math.univ-lyon1.fr}}
\date{}

\begin{document}

\maketitle

\selectlanguage{english}

\begin{abstract}
The set of hypertrees on $n$ vertices can be endowed with a poset structure. J. McCammond and J. Meier computed the dimension of the unique non zero homology group of the hypertree poset. We give another proof of their result and use the theory of species to determine the action of the symmetric group on this homology group, which is linked with the anti-cyclic structure of the $\operatorname{Prelie}$ operad. We also compute the action on the Whitney homology of the poset.
\end{abstract}

\selectlanguage{french}

\begin{abstract}
L'ensemble des hyperarbres à $n$ sommets peut être muni d'un ordre partiel. J. McCammond et J. Meier ont calculé la dimension de l'unique groupe d'homologie non trivial du poset des hyperarbres. Après avoir donné une autre preuve de ce résultat, nous utilisons la théorie des espèces pour déterminer l'action du groupe symétrique sur ce groupe, que nous relions à la structure anti-cyclique de l'opérade $\operatorname{Prelie}$. Nous calculons aussi l'action du groupe symétrique sur l'homologie de Whitney du poset.

\end{abstract}

\selectlanguage{english}

\tableofcontents

\section*{Introduction}

	The notion of hypertree has been introduced by C. Berge \cite{Berge} during the 1980's, as a generalization of trees whose edges can contain more than two vertices. Several studies on hypertrees have been led such as the computation of the number of hypertrees on $n$ vertices by L. Kalikow in \cite{Kali} and by Smith and D. Warme in \cite{Warm}. For a finite set $I$, we can endow the set of hypertrees on the vertex set $I$ with a structure of poset: given two hypertrees $H$ and $K$, $H \preceq K$ if each edge in $K$ is a subset of some edge in $H$. These hypertree posets have been used for the study of automorphisms of free groups and free products in papers of D. McCullough-A. Miller \cite{McCuMi}, N. Brady-J. McCammond-J. Meier-A. Miller \cite{BMcCMM}, J. McCammond-J. Meier \cite{McCM} and C. Jensen-J. McCammond-J. Meier \cite{JMcCM} and \cite{JMcCM2}. In the article \cite{BMcCMM}, the Cohen-Macaulayness of the poset is proven: the poset has only one non trivial homology group. The reduced Euler characteristic of the poset have then been computed in \cite{McCM}: the unique non trivial homology group has its dimension equals to $(n-1)^{n-2}$. Taking the set $\{1, \ldots, n\}$ for $I$, the action of the symmetric group $\mathfrak{S}_n$ on $I$ induces an action on the poset of hypertrees on $I$ compatible with the differential: this provides an action of the symmetric group on the unique non trivial homology group of the poset. In the article \cite{ChHyp}, F. Chapoton computed the characteristic polynomial of the poset and gave a conjecture for the representation of the symmetric group $\mathfrak{S}_n$ on the homology and on the Whitney homology of the poset.

	This article solves the conjecture of F. Chapoton in theorems \ref{Zmu} and \ref{goalt}. The dimension computed by J. McCammond and J. Meier turns out to be also the number of labelled rooted trees on $n-1$ vertices, which is the dimension of the vector space $\operatorname{PreLie}(n-1)$, the component of arity n-1 of the PreLie operad. As the operad  $\operatorname{PreLie}$ is an anti-cyclic operad, as proven in \cite{ChAOp}, the action of the symmetric group $\mathfrak{S}_{n-1}$ on $\operatorname{PreLie}(n-1)$ induces an action of $\mathfrak{S}_n$ on $\operatorname{PreLie}(n-1)$. In theorem \ref{Zmu}, we prove that the representation of $\mathfrak{S}_n$ on $\operatorname{PreLie}(n-1)$ and the representation of $\mathfrak{S}_n$ on the poset homology are isomorphic up to
tensor product by the sign representation. The theorem \ref{goalt} is a refinement of this theorem in which appears a type of hypertrees decorated by $\Sigma \operatorname{Lie}$: the action of the symmetric group on the unique non trivial homology group of the poset is the same as the action of the symmetric group on these decorated hypertrees.

We recommend to read appendix \ref{rappel espèces} and the first two chapters of the book \cite{BLL} for an introduction to species theory which will be used in the article. In the first part of the article, we recall the construction of the homology group of a poset. In the second part, we determine relations between hypertree and pointed hypertree species and then, give a new proof for J. McCammond and J. Meier's result on the dimension of the poset homology group in the third part. In the fourth part, we use the relations between species, established in the first part, to compute the action of the symmetric group on this homology group. In the last part, we compute the action of the symmetric group on Whitney homology.

\section{Construction of the homology of the hypertree poset}
	
	\subsection{Definition of the poset}
	
	The hypertrees and the associated poset are described by F. Chapoton in the article \cite{ChHyp}. We briefly recall their definitions.
	
		\subsubsection{Hypergraphs and hypertrees}

		\begin{defi}
		 An \textbf{hypergraph (on a set $V$)} is an ordered pair $(V,E)$ where $V$ is a finite set and $E$ is a collection of elements of cardinality at least two, belonging to the power set $\mathcal{P}(V)$. The elements of $V$ are called \textbf{vertices} and these of $E$ are called \textbf{edges}.
		 \end{defi}

		\begin{exple}
		  An example of hypergraph on $\{1,2,3,4,5,6,7\}$:

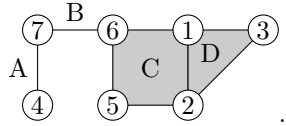
\begin{center}
\begin{tikzpicture}[scale=1]
\draw (0,0) -- (0,1) node[midway, left]{A};
\draw (0,1) -- (1,1) node[midway, above]{B};
\draw[black, fill=gray!40] (1,1) -- (1,0) -- (2,0) -- (2,1) -- (1,1);
\draw(1.5,0.5) node{C};
\draw[black, fill=gray!40] (2,0) -- (2,1)--(3,1) -- (2,0);
\draw(2.30,0.70) node{D};
\draw[black, fill=white] (0,0) circle (0.2);
\draw[black, fill=white] (0,1) circle (0.2);
\draw[black, fill=white] (1,1) circle (0.2);
\draw[black, fill=white] (1,0) circle (0.2);
\draw[black, fill=white] (2,0) circle (0.2);
\draw[black, fill=white] (2,1) circle (0.2);
\draw[black, fill=white] (3,1) circle (0.2);
\draw(0,0) node{$4$};
\draw(0,1) node{$7$};
\draw(1,1) node{$6$};
\draw(1,0) node{$5$};
\draw(2,1) node{$1$};
\draw(2,0) node{$2$};
\draw(3,1) node{$3$};
\end{tikzpicture}.

\end{center}

		\end{exple}

		\begin{defi} Let $H=(V,E)$  be a hypergraph.
		
		 A \textbf{walk from a vertex or an edge $d$ to a vertex or an edge $f$ in $H$} is an alternating sequence of vertices and edges beginning by $d$ and ending by $f$ $(d, \ldots, e_i, v_i, e_{i+1}, \ldots, f)$ where for all $i$, $v_i \in V$, $e_i \in E$ and $\{v_i,v_{i+1}\} \subseteq e_i$. The \textit{length} of a walk is the number of edges and vertices in the walk.
		\end{defi}

\begin{exple} In the previous example, there are several walks from $4$ to $2$: $(4,A,7,B,6,C,2)$ and $(4,A,7,B,6,C,1,D,3,D,2)$. A walk from $C$ to $3$ is $(C,1,D,3)$
\end{exple}

\begin{defi} An \textbf{hypertree} is a non empty hypergraph $H$ such that, given any vertices $v$ and $w$ in $H$, 
\begin{itemize}
\item there exists a walk from $v$ to $w$ in $H$ with distinct edges $e_i$, i.e. $H$ is  \emph{connected},
\item and this walk is unique, i.e. $H$ has \emph{no cycles}. 
 \end{itemize}

The pair $H=(V,E)$ is called \emph{hypertree on $V$}. If $V$ is the set $ \{ 1, \ldots, n \}$, then $H$ is called an \emph{hypertree on $n$ vertices}.
\end{defi}

Denote the hypertree species by $\mathcal{H}$.

\begin{exple}
An example of hypertree on $\{1,2,3,4\}$: 

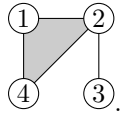
\begin{center}
\begin{tikzpicture}[scale=1]
\draw[black, fill=gray!40] (0,0) -- (1,1) -- (0,1) -- (0,0);
\draw(1,0) -- (1,1);
\draw[black, fill=white] (0,0) circle (0.2);
\draw[black, fill=white] (0,1) circle (0.2);
\draw[black, fill=white] (1,1) circle (0.2);
\draw[black, fill=white] (1,0) circle (0.2);
\draw(0,0) node{$4$};
\draw(0,1) node{$1$};
\draw(1,1) node{$2$};
\draw(1,0) node{$3$};
\end{tikzpicture}.
\end{center}

\end{exple}

We have the following proposition:
\begin{prop} \label{min walk}
Given a hypertree $H$, a vertex or an edge $d$ of $H$ and a vertex $f$ of $H$ , there is a unique minimal walk from $d$ to $f$ and this walk have distinct edges.
\end{prop}

\begin{proof}
If $d$ is a vertex, there exists a unique walk to $f$ with distinct edges as $H$ is a hypertree. Let us consider another walk $(d=v_0, e_1, v_1, \ldots, e_k, v_k=f)$ with $e_i=e_j$ for some $i<j$. Then $(d=v_0, \ldots, e_i, v_j, \ldots, v_k=f)$, obtained by deleting $(v_i, \ldots, e_j)$ in the walk, is a shorter walk. Then, a minimal walk have distinct edges and is thus unique.

If $d$ is an edge, we consider $(v,v')$ a pair of vertex in $d$. If $d$ is not on the unique minimal walk $(v=v_0, e_1, v_1, \ldots, v_n=f)$ from $v$ to $f$, then $(v',d,v_0,e_1, \ldots, v_n=f)$ is a walk from $v'$ to $f$ with distinct edges so it is the unique minimal walk from $v'$ to $f$. Otherwise, let us exchange $v$ and $v'$ so that the edge $d$ is the first edge on the unique minimal walk from $v'$ to $f$. This walk give a walk $w$ from $d$ to $f$ by deleting the vertex $v'$. Suppose that there is another walk different from $w$ from $d$ to $f$ of length less or equal to the length of $w$. By adding $v'$ at the beginning of the walk, this give a walk from $v'$ to $f$ of length less or equal to the length of the unique minimal walk from $v'$ to $f$, and different from it: this is not possible. Thus, there is a unique minimal walk from $d$ to $f$ and this walk have distinct edges.
\end{proof}
	
		\subsubsection{The hypertree poset on $n$ vertices}

Let $I$ be a finite set of cardinality $n$, $S$ and $T$ be two hypertrees on $I$. We say that $S \preceq T$ if each edge of $S$ is the union of edges of $T$, and that $S \prec T$ if $S  \preceq T$ but $S \neq T$.

\begin{exple}
Example in the hypertree poset on four vertices on $I=(\diamondsuit,\heartsuit,\clubsuit,\spadesuit)$: 

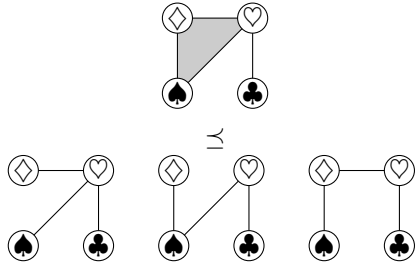
\begin{figure}[!h]
  \begin{center}  
  
  \begin{tikzpicture}[scale=1]
\draw[black, fill=gray!40] (0,0) -- (1,1) -- (0,1) -- (0,0);
\draw(1,0) -- (1,1);
\draw[black, fill=white] (0,0) circle (0.2);
\draw[black, fill=white] (0,1) circle (0.2);
\draw[black, fill=white] (1,1) circle (0.2);
\draw[black, fill=white] (1,0) circle (0.2);
\draw(0,0) node{$\spadesuit$};
\draw(0,1) node{$\diamondsuit$};
\draw(1,1) node{$\heartsuit$};
\draw(1,0) node{$\clubsuit$};
\end{tikzpicture}

  $\preceq$
  
  \begin{tikzpicture}[scale=1]
\draw(0,0) -- (1,1);
\draw(0,1) -- (1,1);
\draw(1,0) -- (1,1);
\draw[black, fill=white] (0,0) circle (0.2);
\draw[black, fill=white] (0,1) circle (0.2);
\draw[black, fill=white] (1,1) circle (0.2);
\draw[black, fill=white] (1,0) circle (0.2);
\draw(0,0) node{$\spadesuit$};
\draw(0,1) node{$\diamondsuit$};
\draw(1,1) node{$\heartsuit$};
\draw(1,0) node{$\clubsuit$};

\draw(2,0) -- (2,1);
\draw(2,0) -- (3,1);
\draw(3,0) -- (3,1);
\draw[black, fill=white] (2,0) circle (0.2);
\draw[black, fill=white] (2,1) circle (0.2);
\draw[black, fill=white] (3,1) circle (0.2);
\draw[black, fill=white] (3,0) circle (0.2);
\draw(2,0) node{$\spadesuit$};
\draw(2,1) node{$\diamondsuit$};
\draw(3,1) node{$\heartsuit$};
\draw(3,0) node{$\clubsuit$};

\draw(4,0) -- (4,1);
\draw(4,1) -- (5,1);
\draw(5,0) -- (5,1);
\draw[black, fill=white] (4,0) circle (0.2);
\draw[black, fill=white] (4,1) circle (0.2);
\draw[black, fill=white] (5,1) circle (0.2);
\draw[black, fill=white] (5,0) circle (0.2);
\draw(4,0) node{$\spadesuit$};
\draw(4,1) node{$\diamondsuit$};
\draw(5,1) node{$\heartsuit$};
\draw(5,0) node{$\clubsuit$};
\end{tikzpicture}.

  \end{center}
\end{figure} 
\end{exple}

The set $(\mathcal{H}(I),\preceq)$ is a partially ordered set (or poset), written $\operatorname{HT(I)}$. We denote by $\widehat{\operatorname{HT(I)}}$ the poset obtained by adding to $\operatorname{HT(I)}$ a formal element $\hat{1}$ above all the other elements of the poset. We moreover write $\operatorname{HT_n}$ for the poset $\operatorname{HT(\{1, \ldots, n\})}$.

\begin{defi} Given a relation $\preceq$, the \emph{cover relation} $\vartriangleleft$ is defined by $x \vartriangleleft y$ ($y$ covers $x$ or $x$ is covered by $y$) if and only if $x \prec y$ and there is no $z$ such that $x \prec z \prec y$ .
\end{defi}

In $\operatorname{HT(I)}$, we define the \emph{rank} $r(h)$ of a hypertree $h$ with $A$ edges by: 
\begin{equation*}
r(h)=A-1.
\end{equation*}

Each cover relation increases the rank by one, so the poset $\operatorname{HT(I)}$ is graded by the number of edges in hypertrees.

	\subsection{Chain complex and homology associated to a poset}

We now define the homology associated to a poset $\mathcal{P}$ with a minimum and a maximum. The reader may read Wachs' article \cite{Wachs} for a deeper treatment of this subject and Munkres' book \cite{M} for more details on simplicial homology. We introduce the following terminology: 

\begin{defi}
A \emph{strict $m$-chain} is an $m$-tuple $(a_1,\ldots, a_m)$ where $a_i$ are elements of $\mathcal{P}$, neither maximum nor minimum in $\mathcal{P}$,  and $a_i \prec a_{i+1}$, for all $i\geq 1$. We write  $\mathcal{C}_m(\mathcal{P})$ for the set of strict $m+1$-chains and $C_m(\mathcal{P})$ for the vector space generated by all strict $m+1$-chains.
\end{defi}

The set $\cup_{m \geq 0}{\mathcal{C}_{m}(\mathcal{P})}$ is then a simplicial complex. 

Define the linear map $d_m:C_{m+1}(\mathcal{P}) \rightarrow C_{m}(\mathcal{P})$ which maps a $m+1$-simplex to its boundary. These maps satisfy $d_{m-1} \circ d_m=0$. The pairs $(C_m(\mathcal{P}),d_m)_{m>0}$ obtained form a \emph{chain complex}. Thus, we can define the homology of the poset.

\begin{defi}The homology group of dimension $m$ of the poset $\mathcal{P}$ is: 
\begin{equation*}
H_m(\mathcal{P})=\operatorname{ker} d_m /\operatorname{Im} d_{m+1}.
\end{equation*}
\end{defi}

We consider in this article the reduced homology, written $\tilde{H}_i$. Having $C_{-1}(\mathcal{P})=\mathbb{C} . e$, and $d: C_0 \rightarrow C_{-1}$, the trivial linear map which maps every singleton to the element $e$, we obtain: 
\begin{equation*} \operatorname{dim} (\tilde{H_0}(\mathcal{P}))=\operatorname{dim} (H_0(\mathcal{P})-1).
\end{equation*}

Dimensions of the homology spaces satisfy the following well-known property: 

\begin{lem} The Euler characteristic of the homology satisfies: 
\begin{equation}\label{CarEuler1}
\chi =\sum_{m \geq 0} (-1)^{m} \operatorname{dim} \tilde{H}_m(\mathcal{P}) = \sum_{m \geq -1} (-1)^{m} \operatorname{dim} C_m(\mathcal{P}).
\end{equation}
\end{lem}

	\subsection{Homology of the $\widehat{\operatorname{HT_n}}$ poset}

Let us apply the previous subsection to the poset $\widehat{\operatorname{HT_n}}$. The vector spaces $ C_m(\mathcal{P})$ and $\tilde{H}_m(\mathcal{P})$ are denoted by $C_m^n$ and $\tilde{H}_m^n$.The reader may consult Sundaram's article \cite {Sundaram} for general points on the notion of Cohen-Macaulay poset. The following notion is needed: 

\begin{defi} Let $\mathcal{P}$ be a poset and $\sigma$ be a closed simplex of the geometric realization  $|\mathcal{P}|$ of $\mathcal{P}$ . The \emph{link} of $\sigma$ is the subcomplex: 
\begin{equation*}
Lk(\sigma)=\{\lambda \in |\mathcal{P}|: \lambda \cup \sigma \in |\mathcal{P}|, \lambda \cap \sigma = \emptyset\}.
\end{equation*}
\end{defi}

\begin{defi}\cite[definition 2.8]{McCM} A poset $\mathcal{P}$ is \emph{Cohen-Macaulay} if its geometric realization $|\mathcal{P}|$ is Cohen-Macaulay. That is, for every closed simplex $\sigma$ in $|\mathcal{P}|$, we have: 

\begin{equation*}
\tilde{H}_i(Lk(\sigma))=\left\lbrace
\begin{array}{ll}
0, & \text{for }  i \neq \operatorname{dim} (|\mathcal{P}|)-\operatorname{dim}(\sigma)-1\\
\text{torsion free}, & \text{for } i = \operatorname{dim}(|\mathcal{P}|)-\operatorname{dim}(\sigma)-1,
\end{array}\right.
\end{equation*}

where the dimension of the empty simplex is $-1$ by convention.
\end{defi}

\begin{thm} \cite[theorem 2.9]{McCM} \label{Cohen-Macaulay}
For each $n \geq 1$, the poset $\widehat{\operatorname{HT_n}}$ is Cohen-Macaulay.
\end{thm}

\begin{cor}
The homology of $\widehat{\operatorname{HT_n}}$ is concentrated in maximal degree: 
\begin{equation*}
\tilde{H}_i(Lk(\emptyset))=\left\lbrace
\begin{array}{ll}
0, & \text{for } i \neq \operatorname{dim}(|\mathcal{\widehat{\operatorname{HT_n}}}|)\\
\text{torsion free}, & \text{for } i = \operatorname{dim} (|\mathcal{\widehat{\operatorname{HT_n}}}|).
\end{array}\right.
\end{equation*}
\end{cor}

The equation \eqref{CarEuler1} can thus be rewritten as: 

\begin{equation}
\operatorname{dim} \tilde{H}_{n}^n=\sum_{m \geq -1} (-1)^{m} \operatorname{dim} C_m^n.
\end{equation}

Moreover, as the differential is compatible with the symmetric group action, the action of the symmetric group on $(C_m)_{m \geq -1}$ induces an action on $\tilde{H}_{n}^n$. Hence the following relation holds, with $\chi_{i+1}^s$ the character of the action of the symmetric group on the vector space $C_i^n$ and $\chi_{\tilde{H}_{n}^n}$ the character of the action of the symmetric group on the vector space $\tilde{H}_{n}^n$: 

\begin{equation}\label{CarEuler}
\chi_{\tilde{H}_{n}^n} =\sum_{m \geq -1} (-1)^{m} \chi_{i+1}^s.
\end{equation}

	\subsection{From large to strict chains} 
	
	\label{chLarg}

According to equation \eqref{CarEuler}, it is sufficient to compute the alternating sum of characters on $C_m^n$ to determine the character on the only non trivial homology group.

Let $k$ be a natural number and $I$ be a finite set. The set of \emph{large $k$-chains} of hypertrees on $I$ is the set $\operatorname{HL^I_k}$ of $k$-tuples $(a_1,\ldots, a_k)$ where $a_i$ are elements of ${\operatorname{HT(I)}}$ and $a_i \preceq a_{i+1}$. The set of \emph{strict $k$-chains} of hypertrees on  $I$ is the set $\operatorname{HS}^I_k$ of $k$-tuples $(a_1,\ldots, a_k)$ where $a_i$ are non minimum elements of ${\operatorname{HT(I)}}$ and $a_i \prec a_{i+1}$. The set $\operatorname{HS}^{\{1, \ldots, n\}}_k$ is then a basis of the vector space $C^n_{k+1}$.

We define the following species: 

\begin{defi}
The species $\mathcal{H}_k$ of large $k$-chains of hypertrees is defined by: 
\begin{equation*}
I \mapsto \operatorname{HL}^{ I}_k.
\end{equation*}

The species $\mathcal{HS}_k$ of strict $k$-chains of hypertrees is the species defined by: 
\begin{equation*}
I \mapsto \operatorname{HS}^{ I}_k.
\end{equation*}
\end{defi}
 
\begin{defi}
Let ${M}_{k,s}$ be the set of words on $\{0,1 \}$ of length $k$, containing $s$ letters "$1$". The species $\mathcal{M}_{k,s}$ is defined by:

\begin{equation*} 
\left\lbrace
\begin{array}{rcl}
\emptyset & \mapsto & {M}_{k,s}, \\
V \neq \emptyset & \mapsto & \emptyset .
\end{array} 
\right.
\end{equation*} 

\end{defi}

Let us describe the link between these species: 

\begin{prop} \label{prop prec}
The species $\mathcal{H}_k$ and $\mathcal{HS}_i$ are related by: 
\begin{equation*}
\mathcal{H}_k \cong \sum_{i \geq 0} \mathcal{HS}_i \times \mathcal{M}_{k,i}.
\end{equation*}
\end{prop}

\begin{proof}
Let $(a_1,\ldots, a_k)$ be a large $k$-chain. It can be factorized into an ordered pair formed by a strict $s$-chain $(a_{i_1}, \ldots, a_{i_s})$, obtained by deleting repetitions and minimum $\hat{0}$, if it is possible, and an element $u_1 \ldots u_k$ of ${M}_{k,s}$ such that:
\begin{itemize} 
\item $u_1=0$ if $a_1=\hat{0}$, $1$ otherwise;
\item $u_j=0$ if $a_j \neq a_{j-1}$, $1$ otherwise.
\end{itemize}

From a strict $i$-chain and a word $u_1 \ldots u_k$ of ${M}_{k,i}$, a large $k$-chain can be reconstructed.

This establishes the desired species isomorphism.
\end{proof}

\begin{cor}
Consider the action by permutation of $\mathfrak{S}_n$ on $\{1, \ldots, n \}$. The characters $\chi_k$ and $\chi^s_i$ of the induced action on the vector spaces $\mathcal{H}_k (\{1, \ldots, n \})$ and  $\mathcal{HS}_i (\{1, \ldots, n \})$ satisfy: 
\begin{equation}
\chi_k= \sum_{i = 0}^{n-2} \binom{k}{i} \chi^s_i.
\end{equation}
\end{cor}

\begin{proof}
The isomorphism of proposition \ref{prop prec} is a species isomorphism, so it preserves the symmetric group action.

This gives: 
\begin{equation*}
\mathcal{H}_k(\{1, \ldots, n \}) \cong \sum_{i} \mathcal{HS}_i(\{1, \ldots, n \}) \times \mathcal{M}_{k,i} (\emptyset).
\end{equation*}

Moreover, the action of $\mathfrak{S}_n$ on $\mathcal{M}_{k,i} (\emptyset)$ is trivial, so that we obtain:  
\begin{equation*}
\chi_k= \sum_{i \geq 0} \chi^s_i \times \# {M}_{k,i}.
\end{equation*}

The cardinality of ${M}_{k,i}$ is $\binom{k}{i}$. As the maximal length of a strict chain in $\operatorname{HT_n}$ is $n-2$, the sum is finite.

\end{proof}

As the expression of $\sum_{i = 0}^{n-2} \binom{k}{i} \chi^s_i$ is polynomial in $k$, of degree bounded by $n$, it enables us to extend $\chi_k$ to integers. Equation \eqref{CarEuler} shows that the character $\chi_k$ evaluated at $k=-1$ is the opposite of the character given by the action of $\mathfrak{S}_n$ induced on poset homology. 

\begin{prop}\label{mu}
Let us write $P_n(X)$ for the polynomial whose value in $k$ gives the number of large $k$-chains in the poset $\widehat{HT_n}$.
The opposite of the character given by the action of $\mathfrak{S}_n$ induced on the homology of poset $\widehat{HT_n}$ is given by $P_n(-1)$. 
\end{prop}

\section{Relations between species and auxiliary species}

In this section, we define new species and establish connections between them. The reader may consult the appendix \ref{rappel espèces} for definitions of some usual species used in this part.

	\label{rel esp}

	\subsection{Pointed hypertrees}
	
Let $k$ be a natural number.

We define the following pointed hypertrees: 

\begin{defi} A \textbf{rooted hypertree} is a hypertree $H$ together with a vertex $s$ of $H$. The hypertree $H$ is said to be \emph{rooted at $s$} and $s$ is called the root of $H$.
\end{defi}

\begin{exple} A hypertree on nine vertices, rooted at $1$.

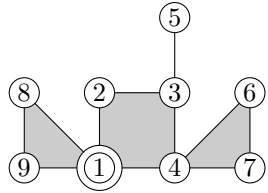
\begin{center}
\begin{tikzpicture}[scale=1]
\draw[black, fill=gray!40] (1,1) -- (1,0) -- (2,0) -- (2,1) -- (1,1);
\draw[black, fill=gray!40] (2,0) -- (3,0)--(3,1) -- (2,0);
\draw[black, fill=gray!40] (0,0) -- (1,0)--(0,1) -- (0,0);
\draw(2,2) -- (2,1);
\draw[black, fill=white] (0,0) circle (0.2);
\draw[black, fill=white] (0,1) circle (0.2);
\draw[black, fill=white] (1,1) circle (0.2);
\draw[black, fill=white] (1,0) circle (0.3);
\draw[black, fill=white] (1,0) circle (0.2);
\draw[black, fill=white] (2,0) circle (0.2);
\draw[black, fill=white] (2,1) circle (0.2);
\draw[black, fill=white] (3,1) circle (0.2);
\draw[black, fill=white] (3,0) circle (0.2);
\draw[black, fill=white] (2,2) circle (0.2);
\draw(0,0) node{$9$};
\draw(0,1) node{$8$};
\draw(1,1) node{$2$};
\draw(1,0) node{$1$};
\draw(2,1) node{$3$};
\draw(2,0) node{$4$};
\draw(3,1) node{$6$};
\draw(2,2) node{$5$};
\draw(3,0) node{$7$};
\end{tikzpicture}

\end{center}
\end{exple}

Let us recall that the minimum of a chain is the hypertree with the smallest number of edges on the chain.

The species associated with rooted hypertrees is denoted by $\mathcal{H}^p$. The one associated with large $k$-chains of hypertrees, whose minimum is a rooted hypertree, is denoted by $\mathcal{H}^p_k$. This vertex is then distinguished in the other hypertrees of the chain, so that all hypertrees in the chain can be considered as rooted at this vertex. In the following, the species $\mathcal{H}^p_k$ will be called "species of large rooted $k$-chains".

\begin{defi}  An \textbf{edge-pointed hypertree} is a hypertree $H$ together with an edge $a$ of $H$. The hypertree $H$ is said to be \emph{pointed at $a$}.
\end{defi}

\begin{exple}  A hypertree on seven vertices, pointed at the edge $\{1,2,3,4\}$.

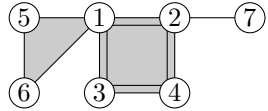
\begin{center}
\begin{tikzpicture}[scale=1]
\draw[black, fill=gray!40] (1,1) -- (1,0) -- (2,0) -- (2,1) -- (1,1);
\draw[black, fill=gray!40] (1.1,0.9) -- (1.1,0.1) -- (1.9,0.1) -- (1.9,0.9) -- (1.1,0.9);
\draw[black, fill=gray!40] (0,0) -- (1,1)--(0,1) -- (0,0);
\draw(3,1) -- (2,1);
\draw[black, fill=white] (0,0) circle (0.2);
\draw[black, fill=white] (0,1) circle (0.2);
\draw[black, fill=white] (1,1) circle (0.2);
\draw[black, fill=white] (1,0) circle (0.2);
\draw[black, fill=white] (2,0) circle (0.2);
\draw[black, fill=white] (2,1) circle (0.2);
\draw[black, fill=white] (3,1) circle (0.2);
\draw(0,0) node{$6$};
\draw(0,1) node{$5$};
\draw(1,1) node{$1$};
\draw(1,0) node{$3$};
\draw(2,1) node{$2$};
\draw(2,0) node{$4$};
\draw(3,1) node{$7$};
\end{tikzpicture}

\end{center}
\end{exple}

The species associated with edge-pointed hypertrees is denoted by $\mathcal{H}^a$. The one associated with large $k$-chains of hypertrees whose minimum is an edge-pointed hypertree is denoted by $\mathcal{H}^a_k$.

\begin{defi} An \textbf{edge-pointed rooted hypertree} is a hypertree $H$ on at least two vertices, together with an edge $a$ of $H$ and a vertex $v$ of $a$. The hypertree $H$ is said to be \emph{pointed at $a$ and rooted at $s$}.
\end{defi}

\begin{exple} A hypertree on seven vertices, pointed at edge $\{1,2,3,4\}$ and rooted at $3$

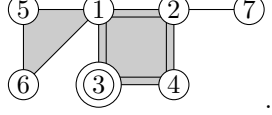
\begin{center}
\begin{tikzpicture}[scale=1]
\draw[black, fill=gray!40] (1,1) -- (1,0) -- (2,0) -- (2,1) -- (1,1);
\draw[black, fill=gray!40] (1.1,0.9) -- (1.1,0.1) -- (1.9,0.1) -- (1.9,0.9) -- (1.1,0.9);
\draw[black, fill=gray!40] (0,0) -- (1,1)--(0,1) -- (0,0);
\draw(3,1) -- (2,1);
\draw[black, fill=white] (1,0) circle (0.3);
\draw[black, fill=white] (0,0) circle (0.2);
\draw[black, fill=white] (0,1) circle (0.2);
\draw[black, fill=white] (1,1) circle (0.2);
\draw[black, fill=white] (1,0) circle (0.2);
\draw[black, fill=white] (2,0) circle (0.2);
\draw[black, fill=white] (2,1) circle (0.2);
\draw[black, fill=white] (3,1) circle (0.2);
\draw(0,0) node{$6$};
\draw(0,1) node{$5$};
\draw(1,1) node{$1$};
\draw(1,0) node{$3$};
\draw(2,1) node{$2$};
\draw(2,0) node{$4$};
\draw(3,1) node{$7$};
\end{tikzpicture}.
\end{center}
\end{exple}

The species associated with edge-pointed rooted hypertrees is denoted by $\mathcal{H}^{pa}$. The one associated with large $k$-chains of hypertrees whose minimum is an edge-pointed rooted hypertree is denoted by $\mathcal{H}^{pa}_k$. 

	\subsection{Dissymmetry principle} 
	
	\label{paragraphe princ de diss}
	
The reader may consult book \cite[Chapitre 2.3]{BLL} for a deeper explanation on the dissymmetry principle. In a general way, a \emph{dissymmetry principle} is the use of a natural center to obtain the expression of a non pointed species in terms of pointed species. An example of this principle is the use of the center of a tree to express unrooted trees in terms of rooted trees. The expression of the hypertree species in term of pointed and rooted hypertrees species is the following:

\begin{prop}
The species of hypertrees and of rooted hypertrees are related by: 
\begin{equation}
\mathcal{H}+\mathcal{H}^{pa}=\mathcal{H}^p+\mathcal{H}^a.
\end{equation}
\end{prop}

\begin{proof}

For the proof, we need the following notions which use the proposition \ref{min walk}: 

\begin{defi} The \emph{eccentricity} of a vertex or an edge is the maximal number of vertices and edges on the minimal walk from it to another vertex. The \emph{center} of a hypertree (edge-pointed or not, rooted or not) is the vertex or the edge with minimal eccentricity.
\end{defi}

\begin{prop}
The center is unique.
\end{prop}

\begin{proof} We prove this proposition \textit{ad absurdum}.

Let us consider a hypertree $H$ such that there are two different vertices or edges $a$ and $b$ of same eccentricity $e$ which are centers of $H$. The number of vertices or edges on a walk from an edge to a vertex is even. The number of vertices or edges on a walk from a vertex to a vertex is odd. Therefore, either $a$ and $b$ are vertices, or they are edges, according to the parity of $e$. As they are different, there is a non trivial minimal walk of odd length from $a$ to $b$ with at least one element $c$ on it different from $a$ and $b$.

We consider a walk $(b, \ldots, e_n, v_n=f)$ from $b$ to a vertex $f$ such that $c$ is not in the walk. If $c$ is not in the unique minimal walk $(a, \ldots, e'_p, v'_p=f)$ from $a$ to $f$, then the concatenation $(b, \ldots, e_n, f, e'_p, \ldots, a)$ is a walk from $b$ to $a$ and $c$ is not in it. The edges of type $e_i$ (respectively $e'_j$) are all different. If this walk is not minimal, there is a minimal $i$ such that $e_i$ and $e'_j$ are equals for an integer $j$. Then the walk $(b, \ldots, e_i, v'_j, \ldots, a)$ is minimal and $c$ is not on it. It means that there are two different minimal walk from $b$ to $a$, which is not possible.

Therefore, for every vertex $f$, $c$ is either in the walk from $b$ to $f$ or in the walk from $a$ to $f$. The eccentricity of $c$ is then strictly less than $e$, which is in contradiction with the minimality of $e$.

\end{proof}

The following maps are bijections, inverse one of each other: 
\begin{equation*}\phi: \mathcal{H}+\mathcal{H}^{pa} \rightarrow \mathcal{H}^{a}+\mathcal{H}^{p},
\end{equation*} 
\begin{equation*}\psi: \mathcal{H}^{a}+\mathcal{H}^{p} \rightarrow \mathcal{H}+\mathcal{H}^{pa}.
\end{equation*}

If $T$ belongs to $\mathcal{H}$, $\phi(T)$ is the hypertree obtained by pointing the center of $T$. We thus obtain a rooted hypertree if the center is a vertex and an edge-pointed hypertree otherwise.  (case A)

Otherwise, $T$ belongs to $\mathcal{H}^{pa}$, $\phi(T)$ is the hypertree obtained from $T$ by: 
\begin{itemize}
\item forgetting the root of $T$ if it is its center, obtaining an edge-pointed hypertree, (case B) 
\item forgetting the pointed edge of $T$ if it is its center, obtaining a rooted hypertree, (case C)
\item forgetting the pointed edge or root which is the nearest from the center of the hypertree.(case D)
\end{itemize}

If $T$ belongs to $\mathcal{H}^{a}$, $\psi(T)$ is the hypertree obtained from $T$ by: 
\begin{itemize}
\item forgetting the pointed edge of $T$ if it is its center,(converse of case A)
\item rooting the center of $T$ if it belongs to the pointed edge of $T$, (converse of case B)
\item rooting the nearest vertex of the pointed edge from the center of $T$. (converse of case D)
\end{itemize}

Otherwise, $T$ belongs to $\mathcal{H}^{p}$, $\psi(T)$ is the hypertree obtained from $T$ by: 
\begin{itemize}
\item forgetting the root of $T$ if it is its center, (converse of case A)
\item pointing the center if it is an edge containing the root of $T$, (converse of case C)
\item pointing the nearest edge containing the root from the center of $T$. (converse of case D)
\end{itemize}

\end{proof}

Let $k$ be a natural number.

The following proposition links large $k$-chains of hypertrees, rooted hypertrees, edge-pointed hypertrees and edge-pointed rooted hypertrees.

\begin{prop}[Dissymmetry principle for hypertrees chains] \label{principe de dissymétrie} The following relation holds: 
\begin{equation}
\mathcal{H}_k+\mathcal{H}^{pa}_k=\mathcal{H}^p_k+\mathcal{H}^a_k.
\end{equation}
\end{prop}

\begin{proof}
We apply the dissymmetry principle to the minimum of the chain.
\end{proof}

\subsection{Relations between species}

		\subsubsection{Relations for $\mathcal{H}^p_k$ }
To determine a functional equation for $\mathcal{H}^p_k$, we introduce another type of hypertree.
		
\begin{defi}
A \emph{hollow hypertree} on $n$ vertices ($n \geq 2$) is a hypertree on the set $\{\#, 1, \ldots, n-1\}$, such that the vertex labelled by $\#$, called the gap, belongs to one and only one edge. 
\end{defi}

\begin{exple}Hollow hypertree on eight vertices.

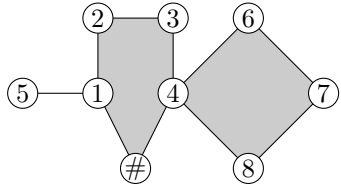
\begin{center}
\begin{tikzpicture}[scale=1]
\draw[black, fill=gray!40] (2,0) -- (3,-1) -- (4,0) -- (3,1) -- (2,0);
\draw[black, fill=gray!40] (1,0) -- (1,1) -- (2,1) -- (2,0) -- (1.5,-1) -- (1,0);
\draw(0,0) -- (1,0);
\draw[black, fill=white] (0,0) circle (0.2);
\draw[black, fill=white] (1,1) circle (0.2);
\draw[black, fill=white] (1,0) circle (0.2);
\draw[black, fill=white] (2,0) circle (0.2);
\draw[black, fill=white] (2,1) circle (0.2);
\draw[black, fill=white] (3,1) circle (0.2);
\draw[black, fill=white] (3,-1) circle (0.2);
\draw[black, fill=white] (4,0) circle (0.2);
\draw[black, fill=white] (1.5,-1) circle (0.2);
\draw(0,0) node{$5$};
\draw(1,1) node{$2$};
\draw(1,0) node{$1$};
\draw(2,1) node{$3$};
\draw(2,0) node{$4$};
\draw(3,1) node{$6$};
\draw(1.5,-1) node{$\#$};
\draw(3,-1) node{$8$};
\draw(4,0) node{$7$};
\end{tikzpicture}
\end{center}
\end{exple}

\begin{defi} \label{rhc} A \textbf{hollow hypertrees $k$-chain} is a chain of length $k$ in the poset of hypertrees on $\{\#, 1, \ldots, n-1\}$, whose minimum is a hollow hypertree. The species of hollow hypertrees $k$-chains is denoted by $\mathcal{H}^c_k$. The species of hollow hypertrees $k$-chains whose minimum has only one edge, is denoted by $\mathcal{H}^{cm}_k$. Remark that the other hypertrees of the chain are not necessarily hollow hypertrees because the vertex labelled by $\#$ is in one and only one edge in the minimum of the chain but can be in two or more edges then.
\end{defi}

These species are linked by the following proposition:

\begin{prop} \label{décomp} The species $ \mathcal{H}^p_k$, $\mathcal{H}^c_k$ and $\mathcal{H}^{cm}_k$ satisfy: 
\begin{equation} \label{epc}
\mathcal{H}^p_k=X \times \operatorname{Comm} \circ  \mathcal{H}^c_k + X,
\end{equation}
\begin{equation} \label{ecct}
\mathcal{H}^c_k= \mathcal{H}^{cm}_k \circ \mathcal{H}^p_k ,
\end{equation}
\begin{equation}\label{ectc}
\mathcal{H}^{cm}_k = \operatorname{Comm} \circ \mathcal{H}^c_{k-1}.
\end{equation}
\end{prop}

\begin{proof}

\begin{enumerate}
\item A $k$-chain of rooted hypertrees on one vertex is just the same as one vertex repeated $k$ times. Thus, it is the same object as a singleton, so the associated species is the species $X$.

We now consider $k$-chains of rooted hypertrees on at least two vertices. Each such chain can be separated into a singleton and a set of hollow hypertrees $k$-chains. The singleton is the root of the minimum hypertree. The set of hollow hypertrees $k$-chains is obtained by: 
  \begin{itemize} 
\item deleting the root in every hypertree,
\item putting a gap $\#$ where the root was,
\item separating in the minimum the edges containing gaps, so that we obtain a set of hollow hypertrees. 
 \end{itemize}

The third point induces a decomposition of a chain into sets of hollow hypertrees chains. Indeed, it gives a partition of the set of edges such that every vertex different from the root appears exactly one time, and this partition is preserved during the chain. This gives the result \eqref{epc}.
 
 \begin{exple} A rooted hypertrees chain decomposed into a singleton and a set of hollow hypertrees $k$-chains. Here are drawn only the minima (at the top) and the maxima (at the bottom) of the chains.
  
 
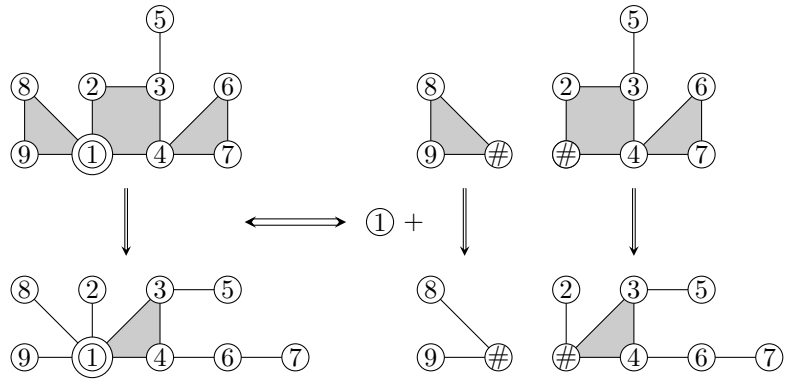
\begin{center}
 \begin{tikzpicture}[scale=0.9]
\draw[black, fill=gray!40] (1,1) -- (1,0) -- (2,0) -- (2,1) -- (1,1);
\draw[black, fill=gray!40] (2,0) -- (3,0)--(3,1) -- (2,0);
\draw[black, fill=gray!40] (0,0) -- (1,0)--(0,1) -- (0,0);
\draw(2,2) -- (2,1);
\draw[black, fill=white] (0,0) circle (0.2);
\draw[black, fill=white] (0,1) circle (0.2);
\draw[black, fill=white] (1,1) circle (0.2);
\draw[black, fill=white] (1,0) circle (0.3);
\draw[black, fill=white] (1,0) circle (0.2);
\draw[black, fill=white] (2,0) circle (0.2);
\draw[black, fill=white] (2,1) circle (0.2);
\draw[black, fill=white] (3,1) circle (0.2);
\draw[black, fill=white] (3,0) circle (0.2);
\draw[black, fill=white] (2,2) circle (0.2);
\draw(0,0) node{$9$};
\draw(0,1) node{$8$};
\draw(1,1) node{$2$};
\draw(1,0) node{$1$};
\draw(2,1) node{$3$};
\draw(2,0) node{$4$};
\draw(3,1) node{$6$};
\draw(2,2) node{$5$};
\draw(3,0) node{$7$};

\draw[double,>=stealth, ->](1.5,-0.5) -- (1.5,-1.5);

\draw[black, fill=gray!40] (1,-3) -- (2,-3) -- (2,-2) -- (1,-3);
\draw(0,-3) -- (1,-3);
\draw(1,-2) -- (1,-3);
\draw(0,-2) -- (1,-3);
\draw(2,-2) -- (3,-2);
\draw(2,-3) -- (3,-3);
\draw(4,-3) -- (3,-3);
\draw[black, fill=white] (1,-3) circle (0.3);
\draw[black, fill=white] (0,-3) circle (0.2);
\draw[black, fill=white] (0,-2) circle (0.2);
\draw[black, fill=white] (1,-2) circle (0.2);
\draw[black, fill=white] (1,-3) circle (0.2);
\draw[black, fill=white] (2,-3) circle (0.2);
\draw[black, fill=white] (2,-2) circle (0.2);
\draw[black, fill=white] (3,-2) circle (0.2);
\draw[black, fill=white] (3,-3) circle (0.2);
\draw[black, fill=white] (4,-3) circle (0.2);
\draw(0,-3) node{$9$};
\draw(0,-2) node{$8$};
\draw(1,-2) node{$2$};
\draw(1,-3) node{$1$};
\draw(2,-2) node{$3$};
\draw(2,-3) node{$4$};
\draw(3,-2) node{$5$};
\draw(3,-3) node{$6$};
\draw(4,-3) node{$7$};

\draw[double distance = 2pt,>=stealth, <->](3.25,-1) -- (4.75,-1);

\draw[black, fill=white] (5.25,-1) circle (0.2);
\draw(5.25,-1) node{$1$};

\draw(5.75,-1) node{$+$};

\draw[black, fill=gray!40] (6,0) -- (7,0)--(6,1) -- (6,0);
\draw[black, fill=white] (6,0) circle (0.2);
\draw[black, fill=white] (6,1) circle (0.2);
\draw[black, fill=white] (7,0) circle (0.2);
\draw(6,0) node{$9$};
\draw(6,1) node{$8$};
\draw(7,0) node{$\#$};

\draw[double,>=stealth, ->](6.5,-0.5) -- (6.5,-1.5);

\draw(6,-3) -- (7,-3);
\draw(6,-2) -- (7,-3);
\draw[black, fill=white] (6,-3) circle (0.2);
\draw[black, fill=white] (6,-2) circle (0.2);
\draw[black, fill=white] (7,-3) circle (0.2);
\draw(6,-3) node{$9$};
\draw(6,-2) node{$8$};
\draw(7,-3) node{$\#$};

\draw[black, fill=gray!40] (8,1) -- (8,0) -- (9,0) -- (9,1) -- (8,1);
\draw[black, fill=gray!40] (9,0) -- (10,0)--(10,1) -- (9,0);
\draw(9,2) -- (9,1);
\draw[black, fill=white] (8,0) circle (0.2);
\draw[black, fill=white] (8,1) circle (0.2);
\draw[black, fill=white] (8,0) circle (0.2);
\draw[black, fill=white] (9,0) circle (0.2);
\draw[black, fill=white] (9,1) circle (0.2);
\draw[black, fill=white] (10,1) circle (0.2);
\draw[black, fill=white] (10,0) circle (0.2);
\draw[black, fill=white] (9,2) circle (0.2);
\draw(8,1) node{$2$};
\draw(8,0) node{$\#$};
\draw(9,1) node{$3$};
\draw(9,0) node{$4$};
\draw(10,1) node{$6$};
\draw(9,2) node{$5$};
\draw(10,0) node{$7$};

\draw[double,>=stealth, ->](9,-0.5) -- (9,-1.5);

\draw[black, fill=gray!40] (8,-3) -- (9,-3) -- (9,-2) -- (8,-3);
\draw(8,-2) -- (8,-3);
\draw(9,-2) -- (10,-2);
\draw(9,-3) -- (10,-3);
\draw(11,-3) -- (10,-3);
\draw[black, fill=white] (8,-2) circle (0.2);
\draw[black, fill=white] (8,-3) circle (0.2);
\draw[black, fill=white] (9,-2) circle (0.2);
\draw[black, fill=white] (9,-3) circle (0.2);
\draw[black, fill=white] (10,-2) circle (0.2);
\draw[black, fill=white] (10,-3) circle (0.2);
\draw[black, fill=white] (11,-3) circle (0.2);
\draw(8,-2) node{$2$};
\draw(8,-3) node{$\#$};
\draw(9,-2) node{$3$};
\draw(9,-3) node{$4$};
\draw(10,-2) node{$5$};
\draw(10,-3) node{$6$};
\draw(11,-3) node{$7$};
\end{tikzpicture}
 \end{center}
 \end{exple}

\item Let $S$ be a hollow hypertrees $k$-chain as defined in definition \ref{rhc}. 

The hollow edge, i.e. the edge containing the gap, gives at each stage $l$ of the chain a set of distinguished edges $D^l_e$. Considering only these distinguished edges, we obtain a hollow hypertrees $k$-chain $D$ whose minimum has only one edge.

Deleting the hollow edge $D^1_e$ in the minimum of $S$ gives a hypertree forest, i.e. a list of hypertrees, $(h_1, \ldots, h_f)$. Each hypertree $h_i$ has a distinguished vertex $s_i$ which was in the hollow edge. Let us say that $h_i$ is rooted at $s_i$. The evolution of edges of the hypertree $h_i$ in $S$ induces a chain $S_{h_i}$. The rootedness of $h_i$ induces a rootedness of $S_{h_i}$.

Note that the hypertrees forest $(h^l_1,\ldots, h^l_f)$ obtained at stage $l$ of the chain by deleting $D^l_e$ is the same as the hypertrees forest obtained by taking the hypertrees at stage $l$ in chains $S_{h_1}, \ldots, S_{h_f}$.

Thus the chain $S$ is chain $D$, where at stage $l$, on vertex $i$, we have grafted hypertree $h_{i,l}$. The grafting consists in replacing vertex $i$ by the root of $h_{i,l}$ in the hypertree.

The chain $S$ can also be seen as chain $D$, where the rooted hypertrees chain $S_{h_i}$ has been inserted in vertex $i$. This gives result \eqref{ecct}.

\begin{exple} Hollow hypertrees chain, separated into a hollow hypertrees chain whose minimum has only one edge, and whose vertices are rooted hypertrees chains.

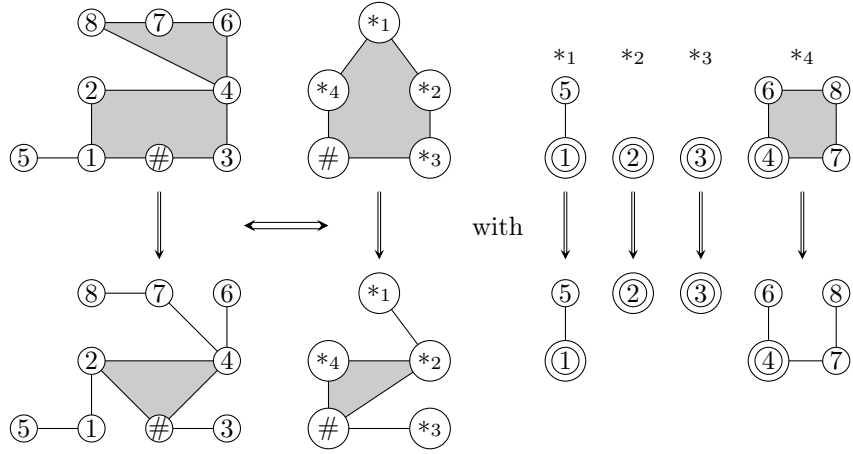
\begin{center}
\begin{tikzpicture}[scale=0.9]
\draw[black, fill=gray!40] (3,1) -- (1,2) -- (2,2) -- (3,2) -- (3,1);
\draw[black, fill=gray!40] (1,0) -- (2,0) -- (3,0) -- (3,1) -- (1,1) -- (1,0);
\draw(0,0) -- (1,0);
\draw[black, fill=white] (0,0) circle (0.2);
\draw[black, fill=white] (1,1) circle (0.2);
\draw[black, fill=white] (1,0) circle (0.2);
\draw[black, fill=white] (2,0) circle (0.2);
\draw[black, fill=white] (3,0) circle (0.2);
\draw[black, fill=white] (3,2) circle (0.2);
\draw[black, fill=white] (3,1) circle (0.2);
\draw[black, fill=white] (2,2) circle (0.2);
\draw[black, fill=white] (1,2) circle (0.2);
\draw(0,0) node{$5$};
\draw(1,1) node{$2$};
\draw(1,0) node{$1$};
\draw(3,0) node{$3$};
\draw(3,1) node{$4$};
\draw(3,2) node{$6$};
\draw(2,0) node{$\#$};
\draw(1,2) node{$8$};
\draw(2,2) node{$7$};

\draw[double,>=stealth, ->](2,-0.5) -- (2,-1.5);

\draw[black, fill=gray!40] (2,-4) -- (1,-3) -- (3,-3) -- (2,-4);
\draw(0,-4) -- (1,-4);
\draw(1,-3) -- (1,-4);
\draw(2,-4) -- (3,-4);
\draw(3,-3) -- (3,-2);
\draw(3,-3) -- (2,-2);
\draw(2,-2) -- (1,-2);
\draw[black, fill=white] (0,-4) circle (0.2);
\draw[black, fill=white] (1,-3) circle (0.2);
\draw[black, fill=white] (1,-4) circle (0.2);
\draw[black, fill=white] (2,-4) circle (0.2);
\draw[black, fill=white] (3,-4) circle (0.2);
\draw[black, fill=white] (3,-2) circle (0.2);
\draw[black, fill=white] (3,-3) circle (0.2);
\draw[black, fill=white] (2,-2) circle (0.2);
\draw[black, fill=white] (1,-2) circle (0.2);
\draw(0,-4) node{$5$};
\draw(1,-3) node{$2$};
\draw(1,-4) node{$1$};
\draw(3,-4) node{$3$};
\draw(3,-3) node{$4$};
\draw(3,-2) node{$6$};
\draw(2,-4) node{$\#$};
\draw(1,-2) node{$8$};
\draw(2,-2) node{$7$};

\draw[double distance = 2pt,>=stealth, <->](3.25,-1) -- (4.5,-1);

\draw[black, fill=gray!40] (5.25,2) -- (6,1) -- (6,0) -- (4.5,0) -- (4.5,1) -- (5.25,2);
\draw[black, fill=white] (5.25,2) circle (0.3);
\draw[black, fill=white] (6,0) circle (0.3);
\draw[black, fill=white] (6,1) circle (0.3);
\draw[black, fill=white] (4.5,0) circle (0.3);
\draw[black, fill=white] (4.5,1) circle (0.3);
\draw(5.25,2) node{$*_1$};
\draw(6,0) node{$*_3$};
\draw(6,1) node{$*_2$};
\draw(4.5,0) node{$\#$};
\draw(4.5,1) node{$*_4$};

\draw[double,>=stealth, ->](5.25,-0.5) -- (5.25,-1.5);

\draw[black, fill=gray!40] (6,-3) -- (4.5,-4) -- (4.5,-3) -- (6,-3);
\draw(6,-3) -- (5.25,-2);
\draw(4.5,-4) -- (6,-4);
\draw[black, fill=white] (5.25,-2) circle (0.3);
\draw[black, fill=white] (6,-4) circle (0.3);
\draw[black, fill=white] (6,-3) circle (0.3);
\draw[black, fill=white] (4.5,-4) circle (0.3);
\draw[black, fill=white] (4.5,-3) circle (0.3);
\draw(5.25,-2) node{$*_1$};
\draw(6,-4) node{$*_3$};
\draw(6,-3) node{$*_2$};
\draw(4.5,-4) node{$\#$};
\draw(4.5,-3) node{$*_4$};

\draw(7,-1) node{with};

\draw(8,1.5) node{$*_1$};

\draw(8,0) -- (8,1);
\draw[black, fill=white] (8,0) circle (0.3);
\draw[black, fill=white] (8,0) circle (0.2);
\draw[black, fill=white] (8,1) circle (0.2);
\draw(8,0) node{$1$};
\draw(8,1) node{$5$};

\draw[double,>=stealth, ->](8,-0.5) -- (8,-1.5);

\draw(8,-2) -- (8,-3);
\draw[black, fill=white] (8,-3) circle (0.3);
\draw[black, fill=white] (8,-3) circle (0.2);
\draw[black, fill=white] (8,-2) circle (0.2);
\draw(8,-3) node{$1$};
\draw(8,-2) node{$5$};

\draw(9,1.5) node{$*_2$};

\draw[black, fill=white] (9,0) circle (0.3);
\draw[black, fill=white] (9,0) circle (0.2);
\draw(9,0) node{$2$};

\draw[double,>=stealth, ->](9,-0.5) -- (9,-1.5);

\draw[black, fill=white] (9,-2) circle (0.3);
\draw[black, fill=white] (9,-2) circle (0.2);
\draw(9,-2) node{$2$};

\draw(10,1.5) node{$*_3$};

\draw[black, fill=white] (10,0) circle (0.3);
\draw[black, fill=white] (10,0) circle (0.2);
\draw(10,0) node{$3$};

\draw[double,>=stealth, ->](10,-0.5) -- (10,-1.5);

\draw[black, fill=white] (10,-2) circle (0.3);
\draw[black, fill=white] (10,-2) circle (0.2);
\draw(10,-2) node{$3$};

\draw(11.5,1.5) node{$*_4$};

\draw[black, fill=gray!40] (11,1) -- (11,0) -- (12,0) -- (12,1) -- (11,1);
\draw[black, fill=white] (11,1) circle (0.2);
\draw[black, fill=white] (11,0) circle (0.3);
\draw[black, fill=white] (11,0) circle (0.2);
\draw[black, fill=white] (12,0) circle (0.2);
\draw[black, fill=white] (12,1) circle (0.2);
\draw(11,1) node{$6$};
\draw(11,0) node{$4$};
\draw(12,1) node{$8$};
\draw(12,0) node{$7$};

\draw[double,>=stealth, ->](11.5,-0.5) -- (11.5,-1.5);

\draw(11,-2) -- (11,-3);
\draw(12,-3) -- (11,-3);
\draw(12,-2) -- (12,-3);
\draw[black, fill=white] (11,-2) circle (0.2);
\draw[black, fill=white] (11,-3) circle (0.3);
\draw[black, fill=white] (11,-3) circle (0.2);
\draw[black, fill=white] (12,-3) circle (0.2);
\draw[black, fill=white] (12,-2) circle (0.2);
\draw(11,-2) node{$6$};
\draw(11,-3) node{$4$};
\draw(12,-2) node{$8$};
\draw(12,-3) node{$7$};

\end{tikzpicture}
\end{center}
\end{exple}
\item A hollow hypertrees $k$-chains, whose minimum has only one edge can be seen as a ($k-1$)-chain $C_{k-1}$ with a vertex labelled by $\#$. Separating the edges containing the label $\#$ in the minimum of $C_{k-1}$ is the same has separating this chain in a non-empty set of hollow hypertrees $(k-1)$-chains. This gives the result \eqref{ectc}.

\end{enumerate}

\end{proof}

As the species $\mathcal{H}^p_k$ can be factored by the species $X$, the map $\frac{\mathcal{H}^p_k - X}{X}$ is a species. We obtain the following corollary:

\begin{cor}The species $\mathcal{H}^p_k$ satisfies: 
\begin{equation}\label{eHpk}
\mathcal{H}^p_k=X \times Comm \circ \left( \frac{\mathcal{H}^p_{k-1}-X}{X} \circ \mathcal{H}^p_k \right) +X.
\end{equation}
\end{cor}

		\subsubsection{Relations for $\mathcal{H}^a_k$}

We have the following relation: 

\begin{prop} \label{eHak} The species $\mathcal{H}^a_k$ satisfies: 
\begin{equation}\mathcal{H}^a_k=(\mathcal{H}_{k-1}-X) \circ \mathcal{H}^p_k.
\end{equation}
\end{prop}

\begin{proof}
Let $S$ be an edge-pointed hypertrees $k$-chain.
The pointed edge in the minimum of $S$ gives at each stage $l$ of the chain a set of distinguished edges $D^l_e$, obtained from the fission of the pointed edge. Considering only these distinguished edges, we obtain a hypertrees $k$-chain $D$ whose minimum has only one edge. That chain can be seen as a ($k-1$)-chain of hypertrees on at least two vertices.

Deleting the pointed edge $D^1_e$ in the minimum of $S$ gives a hypertrees forest, i.e. a list of hypertrees, $(h_1, \ldots, h_f)$. Each hypertree $h_i$ has a distinguished vertex $s_i$ which was in the pointed edge. Let us say that $h_i$ is rooted at $s_i$. The evolution of edges of hypertree $h_i$ at $S$ induces a chain $S_{h_i}$. The rootedness of $h_i$ induces a rootedness of $S_{h_i}$.

Note that the hypertrees forest $(h^l_1,\ldots, h^l_f)$ obtained at stage $l$ of the chain by deleting $D^l_e$ is the same as the hypertrees forest obtained by taking hypertree at stage $l$ in chains $S_{h_1}, \ldots, S_{h_f}$

Thus the chain $S$ is the chain $D$, where at stage $l$, on vertex $i$, we have grafted hypertree $h_{i,l}$. The grafting consists in replacing vertex $i$ by the root of $h_{i,l}$ in the hypertree.

The chain $S$ can also be seen as the chain $D$, where the rooted hypertrees chain $S_{h_i}$ has been inserted in vertex $i$. This gives the result, as in the proof of proposition \ref{décomp}.

\begin{exple} Edge-pointed hypertrees chain, separated into a hypertrees chain whose vertices are rooted hypertrees chains.

\begin{center}
\begin{tikzpicture}[scale=1]
\draw[black, fill=gray!40] (1,1) -- (1,0) -- (2,0) -- (2,1) -- (1,1);
\draw[black, fill=gray!40] (1.1,0.9) -- (1.1,0.1) -- (1.9,0.1) -- (1.9,0.9) -- (1.1,0.9);
\draw[black, fill=gray!40] (0,0) -- (1,1)--(0,1) -- (0,0);
\draw(3,1) -- (2,1);
\draw[black, fill=white] (0,0) circle (0.2);
\draw[black, fill=white] (0,1) circle (0.2);
\draw[black, fill=white] (1,1) circle (0.2);
\draw[black, fill=white] (1,0) circle (0.2);
\draw[black, fill=white] (2,0) circle (0.2);
\draw[black, fill=white] (2,1) circle (0.2);
\draw[black, fill=white] (3,1) circle (0.2);
\draw(0,0) node{$6$};
\draw(0,1) node{$5$};
\draw(1,1) node{$1$};
\draw(1,0) node{$3$};
\draw(2,1) node{$2$};
\draw(2,0) node{$4$};
\draw(3,1) node{$7$};

\draw[double,>=stealth, ->](1.5,-0.5) -- (1.5,-1.5);

\draw(0,-2) -- (1,-2);
\draw(0,-3) -- (1,-2);
\draw(3,-2) -- (2,-2);
\draw(1,-3) -- (1,-2);
\draw(1,-3) -- (2,-3);
\draw(2,-2) -- (2,-3);
\draw[black, fill=white] (0,-3) circle (0.2);
\draw[black, fill=white] (0,-2) circle (0.2);
\draw[black, fill=white] (1,-2) circle (0.2);
\draw[black, fill=white] (1,-3) circle (0.2);
\draw[black, fill=white] (2,-3) circle (0.2);
\draw[black, fill=white] (2,-2) circle (0.2);
\draw[black, fill=white] (3,-2) circle (0.2);
\draw(0,-3) node{$6$};
\draw(0,-2) node{$5$};
\draw(1,-2) node{$1$};
\draw(1,-3) node{$3$};
\draw(2,-2) node{$2$};
\draw(2,-3) node{$4$};
\draw(3,-2) node{$7$};

\draw[double distance = 2pt,>=stealth, <->](3.25,-1) -- (4.75,-1);

\draw[black, fill=gray!40] (5,1) -- (5,0) -- (6,0) -- (6,1) -- (5,1);
\draw[black, fill=gray!40] (5.1,0.9) -- (5.1,0.1) -- (5.9,0.1) -- (5.9,0.9) -- (5.1,0.9);
\draw[black, fill=white] (5,1) circle (0.3);
\draw[black, fill=white] (5,0) circle (0.3);
\draw[black, fill=white] (6,0) circle (0.3);
\draw[black, fill=white] (6,1) circle (0.3);
\draw(5,1) node{$*_1$};
\draw(5,0) node{$*_3$};
\draw(6,1) node{$*_2$};
\draw(6,0) node{$*_4$};

\draw[double,>=stealth, ->](5.5,-0.5) -- (5.5,-1.5);

\draw(5,-2) -- (5,-3);
\draw(6,-3) -- (5,-3);
\draw(6,-2) -- (6,-3);
\draw[black, fill=white] (5,-2) circle (0.3);
\draw[black, fill=white] (5,-3) circle (0.3);
\draw[black, fill=white] (6,-3) circle (0.3);
\draw[black, fill=white] (6,-2) circle (0.3);
\draw(5,-2) node{$*_1$};
\draw(5,-3) node{$*_3$};
\draw(6,-2) node{$*_2$};
\draw(6,-3) node{$*_4$};

\draw(7,-1) node{with};

\draw(8.5,1.5) node{$*_1$};

\draw[black, fill=gray!40] (8,0) -- (9,1)--(8,1) -- (8,0);
\draw[black, fill=white] (9,1) circle (0.3);
\draw[black, fill=white] (8,0) circle (0.2);
\draw[black, fill=white] (8,1) circle (0.2);
\draw[black, fill=white] (9,1) circle (0.2);
\draw(8,0) node{$6$};
\draw(8,1) node{$5$};
\draw(9,1) node{$1$};

\draw[double,>=stealth, ->](8.5,-0.5) -- (8.5,-1.5);

\draw (8,-3) -- (9,-2);
\draw (8,-2) -- (9,-2);
\draw[black, fill=white] (8,-3) circle (0.2);
\draw[black, fill=white] (8,-2) circle (0.2);
\draw[black, fill=white] (9,-2) circle (0.3);
\draw[black, fill=white] (9,-2) circle (0.2);
\draw(8,-3) node{$6$};
\draw(8,-2) node{$5$};
\draw(9,-2) node{$1$};

\draw(10,1.5) node{$*_2$};

\draw(10,0) -- (10,1);
\draw[black, fill=white] (10,0) circle (0.3);
\draw[black, fill=white] (10,0) circle (0.2);
\draw[black, fill=white] (10,1) circle (0.2);
\draw(10,0) node{$2$};
\draw(10,1) node{$7$};

\draw[double,>=stealth, ->](10,-0.5) -- (10,-1.5);

\draw(10,-2) -- (10,-3);
\draw[black, fill=white] (10,-3) circle (0.3);
\draw[black, fill=white] (10,-3) circle (0.2);
\draw[black, fill=white] (10,-2) circle (0.2);
\draw(10,-3) node{$2$};
\draw(10,-2) node{$7$};

\draw(11,1.5) node{$*_3$};

\draw[black, fill=white] (11,0) circle (0.3);
\draw[black, fill=white] (11,0) circle (0.2);
\draw(11,0) node{$3$};

\draw[double,>=stealth, ->](11,-0.5) -- (11,-1.5);

\draw[black, fill=white] (11,-2) circle (0.3);
\draw[black, fill=white] (11,-2) circle (0.2);
\draw(11,-2) node{$3$};

\draw(12,1.5) node{$*_4$};

\draw[black, fill=white] (12,0) circle (0.3);
\draw[black, fill=white] (12,0) circle (0.2);
\draw(12,0) node{$4$};

\draw[double,>=stealth, ->](12,-0.5) -- (12,-1.5);

\draw[black, fill=white] (12,-2) circle (0.3);
\draw[black, fill=white] (12,-2) circle (0.2);
\draw(12,-2) node{$4$};

\end{tikzpicture}
\end{center}
\end{exple}
\end{proof}

		\subsubsection{Relations for $\mathcal{H}^{pa}_k$}

We have: 

\begin{prop}The species $\mathcal{H}^{pa}_k$ satisfies the functional equation: 
\begin{equation}\mathcal{H}^{pa}_k=(\mathcal{H}^p_{k-1} -X) \circ \mathcal{H}^p_k.
\end{equation}
\end{prop}

\begin{proof}
Forgetting the rootedness gives the decomposition of proposition \ref{eHak}.

Rooting edge-pointed hypertrees chain is the same as pointing out a vertex in the hypertrees $k-1$-chain. This gives the result.

\begin{exple}
An edge-pointed rooted hypertrees chain, seen as a rooted hypertrees chain, whose vertices are labelled by rooted hypertrees chains.

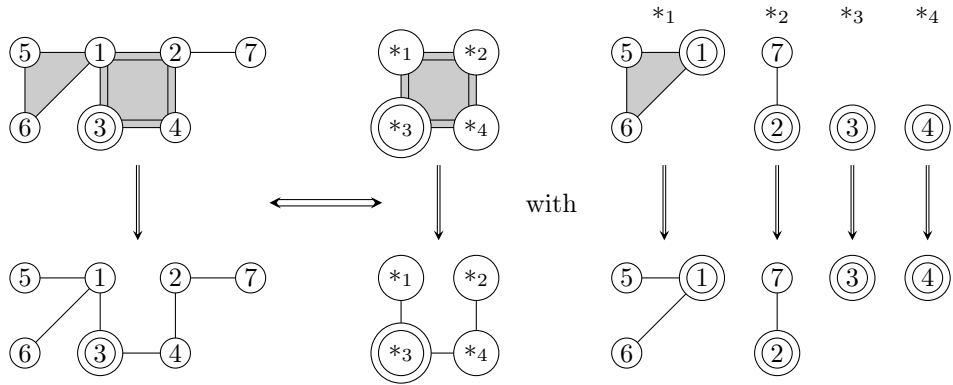
\begin{center}
\begin{tikzpicture}[scale=1]
\draw[black, fill=gray!40] (1,1) -- (1,0) -- (2,0) -- (2,1) -- (1,1);
\draw[black, fill=gray!40] (1.1,0.9) -- (1.1,0.1) -- (1.9,0.1) -- (1.9,0.9) -- (1.1,0.9);
\draw[black, fill=gray!40] (0,0) -- (1,1)--(0,1) -- (0,0);
\draw(3,1) -- (2,1);
\draw[black, fill=white] (1,0) circle (0.3);
\draw[black, fill=white] (0,0) circle (0.2);
\draw[black, fill=white] (0,1) circle (0.2);
\draw[black, fill=white] (1,1) circle (0.2);
\draw[black, fill=white] (1,0) circle (0.2);
\draw[black, fill=white] (2,0) circle (0.2);
\draw[black, fill=white] (2,1) circle (0.2);
\draw[black, fill=white] (3,1) circle (0.2);
\draw(0,0) node{$6$};
\draw(0,1) node{$5$};
\draw(1,1) node{$1$};
\draw(1,0) node{$3$};
\draw(2,1) node{$2$};
\draw(2,0) node{$4$};
\draw(3,1) node{$7$};

\draw[double,>=stealth, ->](1.5,-0.5) -- (1.5,-1.5);

\draw(0,-2) -- (1,-2);
\draw(0,-3) -- (1,-2);
\draw(3,-2) -- (2,-2);
\draw(1,-3) -- (1,-2);
\draw(1,-3) -- (2,-3);
\draw(2,-2) -- (2,-3);
\draw[black, fill=white] (1,-3) circle (0.3);
\draw[black, fill=white] (0,-3) circle (0.2);
\draw[black, fill=white] (0,-2) circle (0.2);
\draw[black, fill=white] (1,-2) circle (0.2);
\draw[black, fill=white] (1,-3) circle (0.2);
\draw[black, fill=white] (2,-3) circle (0.2);
\draw[black, fill=white] (2,-2) circle (0.2);
\draw[black, fill=white] (3,-2) circle (0.2);
\draw(0,-3) node{$6$};
\draw(0,-2) node{$5$};
\draw(1,-2) node{$1$};
\draw(1,-3) node{$3$};
\draw(2,-2) node{$2$};
\draw(2,-3) node{$4$};
\draw(3,-2) node{$7$};

\draw[double distance = 2pt,>=stealth, <->](3.25,-1) -- (4.75,-1);

\draw[black, fill=gray!40] (5,1) -- (5,0) -- (6,0) -- (6,1) -- (5,1);
\draw[black, fill=gray!40] (5.1,0.9) -- (5.1,0.1) -- (5.9,0.1) -- (5.9,0.9) -- (5.1,0.9);
\draw[black, fill=white] (5,0) circle (0.4);
\draw[black, fill=white] (5,1) circle (0.3);
\draw[black, fill=white] (5,0) circle (0.3);
\draw[black, fill=white] (6,0) circle (0.3);
\draw[black, fill=white] (6,1) circle (0.3);
\draw(5,1) node{$*_1$};
\draw(5,0) node{$*_3$};
\draw(6,1) node{$*_2$};
\draw(6,0) node{$*_4$};

\draw[double,>=stealth, ->](5.5,-0.5) -- (5.5,-1.5);

\draw(5,-2) -- (5,-3);
\draw(6,-3) -- (5,-3);
\draw(6,-2) -- (6,-3);
\draw[black, fill=white] (5,-3) circle (0.4);
\draw[black, fill=white] (5,-2) circle (0.3);
\draw[black, fill=white] (5,-3) circle (0.3);
\draw[black, fill=white] (6,-3) circle (0.3);
\draw[black, fill=white] (6,-2) circle (0.3);
\draw(5,-2) node{$*_1$};
\draw(5,-3) node{$*_3$};
\draw(6,-2) node{$*_2$};
\draw(6,-3) node{$*_4$};

\draw(7,-1) node{with};

\draw(8.5,1.5) node{$*_1$};

\draw[black, fill=gray!40] (8,0) -- (9,1)--(8,1) -- (8,0);
\draw[black, fill=white] (9,1) circle (0.3);
\draw[black, fill=white] (8,0) circle (0.2);
\draw[black, fill=white] (8,1) circle (0.2);
\draw[black, fill=white] (9,1) circle (0.2);
\draw(8,0) node{$6$};
\draw(8,1) node{$5$};
\draw(9,1) node{$1$};

\draw[double,>=stealth, ->](8.5,-0.5) -- (8.5,-1.5);

\draw (8,-3) -- (9,-2);
\draw (8,-2) -- (9,-2);
\draw[black, fill=white] (8,-3) circle (0.2);
\draw[black, fill=white] (8,-2) circle (0.2);
\draw[black, fill=white] (9,-2) circle (0.3);
\draw[black, fill=white] (9,-2) circle (0.2);
\draw(8,-3) node{$6$};
\draw(8,-2) node{$5$};
\draw(9,-2) node{$1$};

\draw(10,1.5) node{$*_2$};

\draw(10,0) -- (10,1);
\draw[black, fill=white] (10,0) circle (0.3);
\draw[black, fill=white] (10,0) circle (0.2);
\draw[black, fill=white] (10,1) circle (0.2);
\draw(10,0) node{$2$};
\draw(10,1) node{$7$};

\draw[double,>=stealth, ->](10,-0.5) -- (10,-1.5);

\draw(10,-2) -- (10,-3);
\draw[black, fill=white] (10,-3) circle (0.3);
\draw[black, fill=white] (10,-3) circle (0.2);
\draw[black, fill=white] (10,-2) circle (0.2);
\draw(10,-3) node{$2$};
\draw(10,-2) node{$7$};

\draw(11,1.5) node{$*_3$};

\draw[black, fill=white] (11,0) circle (0.3);
\draw[black, fill=white] (11,0) circle (0.2);
\draw(11,0) node{$3$};

\draw[double,>=stealth, ->](11,-0.5) -- (11,-1.5);

\draw[black, fill=white] (11,-2) circle (0.3);
\draw[black, fill=white] (11,-2) circle (0.2);
\draw(11,-2) node{$3$};

\draw(12,1.5) node{$*_4$};

\draw[black, fill=white] (12,0) circle (0.3);
\draw[black, fill=white] (12,0) circle (0.2);
\draw(12,0) node{$4$};

\draw[double,>=stealth, ->](12,-0.5) -- (12,-1.5);

\draw[black, fill=white] (12,-2) circle (0.3);
\draw[black, fill=white] (12,-2) circle (0.2);
\draw(12,-2) node{$4$};
\end{tikzpicture}
\end{center}
\end{exple}

\end{proof}

		\subsubsection{Relations for $\mathcal{H}_k$}

Rootedness gives the following proposition: 

\begin{prop} The species $\mathcal{H}_k$ satisfies: 
\begin{equation} \label{ediff}
X \times \mathcal{H}_k'=\mathcal{H}^p_k,
\end{equation}
where $'$ is species differentiation.
\end{prop}

	\subsection{Back to strict and large chains}
	
The rootedness of a chain does not change the polynomial nature of the character, shown in section \ref{chLarg}. Consequently, generating series and cycle index of  $\mathcal{H}^{p}$ are polynomial in $k$.

Moreover, as the substitution of formal power series with polynomial coefficients is a formal power series with polynomial coefficients, generating series and cycle indices associated with $\mathcal{H}^a$, $\mathcal{H}^{pa}$, $\mathcal{H}^{c}$ and $\mathcal{H}^{cm}$ are polynomial in $k$.

Consequently, for all considered species, we can take the value of cycle index in $-1$ and this will give the character of symmetric group on the homology associated with pointed hypertrees poset.

\section{Dimension of the poset homology }

Generating series associated with species $\mathcal{H}_k$, $\mathcal{H}_k^p$, $\mathcal{H}_k^a$, $\mathcal{H}_k^{pa}$, $\mathcal{H}_k^c$ and $\mathcal{H}_k^{cm}$ are denoted by $\Sk$, $\Spk$, $\Sak$, $\Spak$, $\Sck$ and $\Scuck$. We compute them here.

	\subsection{Connections between generating series}
	
The equalities between species of part \ref{rel esp} give equalities in terms of generating series: 

\begin{prop}
The series $\Spk$ satisfies: 
\begin{equation}\label{Spk}
\Spk=x \times \exp \left( \frac{\Spkmu \circ \Spk}{\Spk}-1 \right).
\end{equation}
The series $\Sak$ satisfies: 
\begin{equation} \label{rel Sak}
\Sak =(\Skmu-x)(\Spk).
\end{equation}
The series $\Spak$ satisfies: 
\begin{equation} \label{rel Spak}
\Spak=(\Spkmu-x)(\Spk).
\end{equation}
The series $\Sk$ satisfies: 
\begin{equation} \label{ediff}
x \times \Sk'=\Spk.
\end{equation}
Moreover, according to the dissymmetry principle of proposition \ref{principe de dissymétrie}, these series also satisfy:
\begin{equation}
\Sk+\Spak=\Spk+\Sak.
\end{equation}

\end{prop}

	\subsection{Values of the series for $k=0$ and $k=-1$}

		\subsubsection{Computation of $\So$ and $\Spo$}
		
		There is only one hypertrees $0$-chain: the empty chain. This gives: 
		
		\begin{equation}\label{So}
		\So=\sum_{n \geq 1} \frac{x^n}{n!}=e^x-1.
		\end{equation}
		
		Relation \eqref{ediff} gives: 
		
		\begin{equation} \label{Spo}
		\Spo=x e^x.
		\end{equation}
		
		\subsubsection{Computation of $\Smu$}
		
		Using proposition \ref{mu}, it is sufficient to study the value in $-1$ of the polynomial whose value in $k$ gives the number of large $k$-chains to obtain the dimension on the homology group. Therefore we study the value in $-1$ of the exponential generating series whose coefficients are these polynomials.
		The series $\Smu$ is given by the following theorem. This result was first proved by McCammond and Meier in \cite{McCM}. We give here another proof: 
		
		\begin{thm}\cite[theorem 5.1]{McCM} \label{thm smu}
		The dimension of the only non trivial homology group of the poset of hypertrees on $n$ vertices is ${(n-1)}^{n-2}$.
		\end{thm}
		
		\begin{proof}
		
		According to equations \eqref{rel Sak} and \eqref{rel Spak}, applied at $k=0$, the dissymmetry principle of corollary \ref{principe de dissymétrie} is:

		$$\begin{array}{rcl}
		\So-\Spo&=&\Sao-\Spao \\
		&=&(\Smu-\Spmu) \circ \Spo.
		\end{array}$$
		
		With equations \eqref{ediff}, \eqref{So} and \eqref{Spo}, this equality is equivalent to: 
		\begin{equation} \label{eq prec}
(\Smu-x\Smu') \circ xe^x=e^x-xe^x-1.
		\end{equation}
		
		We define a new series:
		\begin{defi}
		Let $\SW$ be the series given by: 
		\begin{equation*}
		\SW(x)=\sum_{n \geq 1} (-1)^{n-1}n^{n-1} \frac{x^n}{n!}.
		\end{equation*}
		This series is the suspension of the generating series $W$ of rooted hypertrees species, associated with the $\operatorname{PreLie}$ operad. It satisfies the following equation, obtained from the decomposition of rooted trees (see \cite[page 2]{BLL} for instance): 
		\begin{equation*}
		\SW(x) e^{\SW(x)}=x.
		\end{equation*}
		We compute its differential: 
		\begin{equation*}
		(\SW)'(x)=\frac{1}{x+e^\SW}.
		\end{equation*}
		\end{defi}

		Composing equation \eqref{eq prec} by $\SW$, we get: 
		\begin{equation*} \label{eSmuSpmu}
		\Smu-x\Smu'=e^\SW-x-1.
		\end{equation*}
	
		To conclude, we need the following lemma:
		
		\begin{lem} \label{thm equa}
		Computing the term $e^\SW-x-1$ gives:
		\begin{equation*}
        e^\SW-x-1=\sum_{n \geq 2}  (-1)^{n-1} (n-1)^{n-1} \frac{x^n}{n!}.
        \end{equation*}
		\end{lem}	
		
		\begin{proof} [Proof of the lemma \ref{thm equa}]
		Both parts of the equation vanish at $0$.
		
		On the one hand, differentiation gives: 
		\begin{equation*}
		(e^\SW-x-1)'=\SW'e^\SW-1=\frac{e^\SW-x-e^\SW}{x+e^\SW}=-x\SW'.
		\end{equation*}
		
		On the other hand, we get: 
		\begin{equation*}
		\left( \sum_{n \geq 2}  (-1)^{n-1} (n-1)^{n-1} \frac{x^n}{n!} \right)'=\sum_{n \geq 2}  (-1)^{n-1} (n-1)^{n-1} \frac{x^{n-1}}{(n-1)!}.
		\end{equation*}
		
		It gives: 
		
		\begin{equation*}
		\left(\sum_{n \geq 2}  (-1)^{n-1} (n-1)^{n-1} \frac{x^n}{n!} \right)'=\sum_{n \geq 1}  (-1)^{n} n^{n} \frac{x^n}{n!}=-x\SW'.
		\end{equation*}
		
		The derivatives of these formal series are the same and they both vanish at $0$, so they are equal.
		\end{proof}	
		
		We conclude thanks to lemma \ref{thm equa}, by considering $\Smu=\sum_{n \geq 1} a_n \frac{x^n}{n!}$. Thus coefficients $a_n$ satisfy, for all integers $n>0$: 
		\begin{equation*}
		a_n-na_n=-(n-1)a_n=(-1)^{n-1}(n-1)^{n-1}.
		\end{equation*}

		\end{proof}
		
		\begin{cor} \label{smusw} The derivative of series $\Smu$ is given by: 
		\begin{equation*}
		(\Smu - x)' = \SW.
		\end{equation*}
		\end{cor}
		
		\begin{proof}
		We differentiate the expression of $\Smu$ obtained in the previous theorem: 
		
		\begin{equation*}
		(\Smu - x)' = \sum_{n \geq 2} (-1)^n (n-1)^{n-2} \frac{x^{n-1}}{(n-1)!}.
		\end{equation*}
		This gives the result.
		\end{proof}
		
		\subsubsection{Back to $\Sao$ and $\Spao$}	
		
		The series $\Sao$ and $\Spao$ are given by the following proposition.
		
		\begin{prop} \label{SaoSpao}
		\begin{enumerate}		
		\item The series $\Sao$ satisfies: 
		\begin{equation} \label{Sao}
		\Sao=\sum_{n \geq 2} (n-1)^2 \frac{x^n}{n!}.
		\end{equation}
		\item The series $\Spao$ satisfies: 
		\begin{equation} \label{Spao}
		\Spao=\sum_{n \geq 2} n(n-1) \frac{x^n}{n!}.
		\end{equation}
		\end{enumerate}
		\end{prop}

		\begin{proof}		
		According to equations \eqref{rel Sak} and \eqref{Spo}, $\Sao$ satisfies: 
		\begin{equation*}		
		\Sao(x)=(\Smu -x) \circ \Spo(x)=(\Smu - x) \circ xe^x.
		\end{equation*}		
		 Differentiating this equality and using corollary \ref{smusw}, we get: 
		 \begin{equation*} (\Sao)'(x)=\SW \circ xe^x \times (x+1)e^x=x(x+1)e^x.
		\end{equation*}		 
		  So it gives:
		  \begin{equation*}(\Sao)'(x)= \sum_{n \geq 1} n^2 \frac{x^n}{n!}.
		  \end{equation*}
		
		As $\Sao(0)=0$, we obtain the first result. 
		
		According to equations \eqref{rel Spak} and \eqref{Spo}, $\Spao$ satisfies: 
		\begin{equation*}
		\Spao(x)=(\Spmu -x) \circ \Spo(x)=(x(\Smu'-1)) \circ xe^x.
		\end{equation*}
		 With corollary  \ref{smusw}, we get: 
		 \begin{equation*} 
		 \Spao(x)=(x\SW) \circ xe^x =x^2e^x.
		 \end{equation*}
		 This gives the second result. 
		
		\end{proof}
		
\section{Action of the symmetric group on the poset homology}
	
The reader may consult the appendix \ref{Annexe ind cycl} for basic definitions on cycle index and the appendix \ref{rappel espèces} for definitions of usual species used in this section and the following.	
	
	\subsection{Description of the action}
	
	Let us consider a hypertree poset on $n$ vertices, as described previously. The symmetric group acts on the set of vertices by permutation. This action preserves number of edges and poset order, so it induces an action on the homology associated with poset $\widehat{\operatorname{HT_n}}$. We will determine in this section the character of this action on poset homology.

	In the following, $\Zk$, $\Zpk$, $\Zak$ and $\Zpak$ will stand for cycle indices associated with species $\mathcal{H}_k$, $\mathcal{H}^p_k$, $\mathcal{H}^a_k$ and $\mathcal{H}^{pa}_k$.
	
	\subsection{Connection between cycle indices}
	
	Relations between species of section \ref{rel esp} give the following proposition: 
	
	\begin{prop} The series $\Zk$, $\Zpk$, $\Zak$ and $\Zpak$ satisfy the following relations: 
	
	\begin{equation} \label{rel diss Z}
	\Zk+\Zpak=\Zpk+\Zak,
	\end{equation}
	
	\begin{equation} \label{Zpk}
	\Zpk = p_1 + p_1 \times \comm \circ \left(\frac{\Zpkmu \circ \Zpk - \Zpk}{\Zpk} \right),
	\end{equation}
	
	\begin{equation} \label{Zak}
	\Zak + \Zpk = \Zkmu \circ \Zpk,
	\end{equation}
	
	\begin{equation} \label{Zpak}
	\Zpak + \Zpk=\Zpkmu \circ \Zpk,
	\end{equation}

and
	
	\begin{equation} \label{der}
	p_1 \frac{\partial \Zk }{\partial p_1} = \Zpk.
	\end{equation}
	
	\end{prop}
	
	This relations holds on $\mathbb{Z}$. Indeed the coefficients of the $x_n$ are polynomial in $k$, so we can extend the previous relations holding on $\mathbb{N}$ to $\mathbb{Z}$.	
	
	\subsection{Computation of the symmetric group character}
	
		\subsubsection{Computation of $\Zmu$}
		
		Using proposition \ref{mu}, it is sufficient to study the value in $-1$ of the polynomial whose value in $k$ gives the character of the action of symmetric group on large $k$-chains to obtain the character on the homology group. Therefore we study the value in $-1$ of the exponential generating series whose coefficients are these polynomials.
		
		The $\operatorname{PreLie}$ operad is anti-cyclic as proven in the article of F. Chapoton \cite{ChAOp}. It means that the usual action of the symmetric group $\mathfrak{S}_n$ on the module $\operatorname{PreLie}(n)$, whose basis is the set of rooted trees, can be extended into an action of the symmetric group $\mathfrak{S}_{n+1}$. We write $M$ for the cycle index associated with this anti-cyclic structure.
		
		The reader may consult the article \cite[part 5.4]{ChHyp} for more information on this series. 
						
		We will prove the following theorem, which describes the action of symmetric group on the homology of hypertree poset in terms of cycle indices associated with the $\operatorname{Comm}$ and $\operatorname{PreLie}$ operads: 
		
		\begin{thm}\label{Zmu}
		The cycle index $\Zmu$, which gives the character of the action of the symmetric group on the homology of the hypertree poset, is related to the cycle index $M$ associated with the anti-cyclic structure of  $\operatorname{PreLie}$ operad by: 
		\begin{equation}
		\Zmu=p_1 - \Sigma M= \comm \circ \spl+ p_1 \left( \spl + 1 \right).
		\end{equation}
		The cycle index $\Zpmu$ is given by: 
		\begin{equation} \label{Zpmu}
		\Zpmu=p_1 \left( \spl + 1 \right).
		\end{equation}
		\end{thm}
		
		\begin{proof}		
		
		We first compute $\Zo$ and $\Zpo$. There is only one $0$-chain: the empty chain. It is fixed by every permutation. A quick computation gives: 
		
		\begin{equation*}
		\Zo = \comm.
		\end{equation*}
		
		We derive from equation \eqref{der}: 
		\begin{equation}\label{Zpo}
		\Zpo=p_1 \frac{\partial \comm }{\partial p_1} = \perm = p_1 (1+\comm).
		\end{equation}

		The equation \eqref{Zpk} gives: 
		\begin{equation*}
		\Zpo = p_1 + p_1 \times \comm \circ \left( \frac{\Zpmu \circ \Zpo - \Zpo}{\Zpo} \right),
		\end{equation*}
		
		so
		\begin{equation*}
		p_1+p_1 \times \comm=p_1+ p_1 \times \comm \circ \left( \frac{\Zpmu \circ \perm - \perm}{\perm} \right).
		\end{equation*}
		
		Recall that $\spl \circ \perm = \perm \circ \spl = p_1$, according to \cite{ChHyp} \footnote{This is a consequence of Koszul duality for operads}.
		
		 We obtain: 
		 \begin{equation*}\spl= \frac{\Zpmu  - p_1}{p_1}, 
         \end{equation*}		 
		 hence the result: 
		
		\begin{equation} \label{Zpmu}
		\Zpmu=p_1 \left( \spl + 1 \right).
		\end{equation}
		
		The dissymmetry equation \eqref{rel diss Z}, combined with relations \eqref{Zak} and \eqref{Zpak} in $k=0$ gives: 
		\begin{equation*}\comm+ \Zpmu \circ \perm - \perm= \perm + \Zmu \circ \perm - \perm.
		\end{equation*}
		
		Composing by $\spl$ and replacing $\Zpmu$ by its expression in equation \eqref{Zpmu}, we obtain: 
		
		\begin{equation*} \Zmu = \comm \circ \spl+ p_1 \left( \spl + 1 \right)  - p_1.
		\end{equation*}
		
		As $(p_1(\comm+1)) \circ \spl = p_1$, 
		
		we thus obtain
\begin{equation*}		\comm \circ \spl= \frac{p_1-\spl}{\spl}.
\end{equation*}
		
		Therefore \begin{equation*} \Zmu = -1+\frac{p_1}{\spl} + p_1 \times \spl.
		\end{equation*}	
		
		According to \cite[equation 50]{ChAOp}, composing by the suspension, we get: 
		\begin{equation*} \Sigma M-1 = -p_1(-1 + \spl +\frac{1}{\spl}),
		\end{equation*}
		
		The result is obtained by using the following equality: 
		\begin{equation*}
		(p_1-\Zmu)-1=p_1-\frac{p_1}{\spl} - p_1 \times \spl.
		\end{equation*}
	
		\end{proof} 
		
		\subsubsection{Back to $\Zao$ and $\Zpao$}
		
		In this part, we refine the results obtained at proposition \ref{SaoSpao}.
		
	\begin{thm}
	Cycle indices associated with species of large $0$-chains, whose minimum is an edge-pointed hypertree and species of large $0$-chains, whose minimum is an edge-pointed rooted hypertree, satisfy: 
	
	\begin{equation}
	\Zao=\comm+(p_1-1) \times \perm,
	\end{equation}
	and
	\begin{equation}
	\Zpao=p_1 \perm.
	\end{equation}
	
	For a cycle index $\mathbf{C}$, we write $(\mathbf{C})_n$ for the part of $\mathbf{C}$ corresponding to a representation of the symmetric group $\mathfrak{S}_n$.
		
	Therefore, for all $n \geq 2$, writing $S^{(n-1,1)}$ for the irreducible representation of the symmetric group $\mathfrak{S}_n$ associated with the partition $(n-1,1)$ of $n$, we obtain: 
	
\begin{enumerate}	
\item	$(\Zao)_n$  is the character of the representation $S^{(n-1,1)} \otimes S^{(n-1,1)}$;
\item $(\Zpao)_n$ is the character of the representation $S^{(n-1,1)} \otimes S^{(n-1,1)} \oplus S^{(n-1,1)}$ .
\end{enumerate}
	\end{thm}
	
	\begin{proof}

The equalities come from relations \eqref{Zak} and \eqref{Zpak}, replacing $\Zpo$ by its expression in equation \eqref{Zpo}, $\Zpmu$ by its expression in equation \eqref{Zpmu} and $\Zmu$ by its expression in theorem \ref{Zmu}. We obtain: 

	\begin{equation}
	(\Zao)_n= \sum_{\lambda \vdash n} \frac{p_\lambda}{z_\lambda}+ \sum_{\lambda \vdash n-2} p_1^2\frac{p_\lambda}{z_\lambda}- \sum_{\lambda \vdash n-1} p_1\frac{p_\lambda}{z_\lambda}
	\end{equation}

Denote now  by $f_\lambda$ the number of fixed points in a permutation of type $\lambda$.

The coefficient in front of $\frac{p_\lambda}{z_\lambda}$ in $(\Zao)_n$ is: 
\begin{equation*}1-f_\lambda +f_\lambda (f_\lambda -1) = (f_\lambda -1)^2.
\end{equation*}
	
	In the same way, we obtain:
	 
	\begin{equation}
	(\Zpao)_n= \sum_{\lambda \vdash n-2} p_1^2\frac{p_\lambda}{z_\lambda}.
	\end{equation}

The coefficient in front of $\frac{p_\lambda}{z_\lambda}$ in $(\Zpao)_n$ is 
\begin{equation*}(f_\lambda -1)^2+f_\lambda -1 = f_\lambda (f_\lambda -1).
\end{equation*}

We conclude thanks to the following lemma:
\begin{lem}
The character of the irreducible representation $S^{(n-1,1)}$ on the conjugacy class $C_\sigma$ is equal to $p-1$, where $p$ is the number of fixed points of every element in $C_\sigma$.
\end{lem}	

Indeed, according to the previous lemma the character of the representation $S^{(n-1,1)} \otimes S^{(n-1,1)}$ on the conjugacy class $C_\sigma$ is equal to $(p-1)^2$, where $p$ is the number of fixed points of every element in $C_\sigma$. This gives the first relation.

The second one is obtained by computing the character of the representation $S^{(n-1,1)} \otimes S^{(n-1,1)} \oplus S^{(n-1,1)}$, equal to $f_\sigma(f_\sigma-1)$ on a conjugacy class whose elements have $f_\sigma$ fixed points.
	
	\end{proof}
	
\begin{proof}[of the lemma]
	The natural representation of $\mathfrak{S}_n$ on $\mathbb{C}^n$ is the direct sum of the trivial representation and the representation $S^{(n-1,1)}$.
	
	The character of this representation on a conjugacy class $C_\sigma$ is equal to the number of fixed points of every element of $C_\sigma$. 
	 
	 The character of the trivial representation is equal to $1$. The result is obtained by difference.
	
	\end{proof}

\section{Action of symmetric group on Whitney homology}

	\subsection{Definition and properties of Whitney homology}
	
The reader may consult the article \cite{Wachs} for definitions and properties of Whitney homology.

\begin{defi}
\emph{Whitney homology} of a poset $\mathcal{P}$ with minimum $\hat{0}$ is the collection of spaces: 
\begin{equation}
\operatorname{WH_i}(P)=\oplus_{x \in P} \tilde{H}_{i-2}([\hat{0},x]), i\geq 2.
\end{equation}
\end{defi}

\begin{thm}\cite{Wachs}
If a poset $P$ is Cohen-Macaulay, its Whitney homology satisfies: 
\begin{equation}
\operatorname{WH_i}(P)=\oplus_{x \in P_{i-1}} \tilde{H}_{i-2}([m,x])
\end{equation}
where $P_{i-1}=\{x \in P|r(x)=i-1\}$ and $r(x)$ is the \emph{rank} of $x$.
\end{thm}

As $\widehat{\operatorname{HT_n}}$ is Cohen-Macaulay, according to theorem \ref{Cohen-Macaulay}, it satisfies the previous theorem.
	
	To compute the Whitney homology of $\widehat{\operatorname{HT_n}}$, we define a weight on large $k$-chains:
	
	\begin{defi}
	The \emph{weight} of a hypertrees chain $S$, denoted by $w(S)$, is: 
	\begin{equation*}
	w(S)=\#\text{edge}(\text{max}(S))-1	\end{equation*}	
	where $\#\text{edge}(\text{max}(S))$ is the number of edges of the maximum in $S$.
	
	Note that in $\widehat{\operatorname{HT_n}}$, the weight of a chain is equal to the rank of its maximum.
	\end{defi}
	
	For $\mathcal{E}$ a species with cycle index $\mathbf{C}$, we will denote by $\mathcal{E}_t$ the associated weighted species with cycle index $\mathbf{C}_t$.
	
	Thus, the species  $\mathcal{H}_{{k,t}}$ is the species which associates to a set $A$ the set of all pairs of large hypertrees $k$-chain with the weight of its maximum. Therefore, we have: 
	\begin{equation*}
	\Zkt = \sum_{n \geq 1} \sum_ {i \geq 0} \chi(\operatorname{HL^n_{k,i}}) t^i \frac{x^n}{n!},
	\end{equation*}
	where $\chi(\operatorname{HL^n_{k,i}})$ is the character given by the action of symmetric group $\mathfrak{S}_n$ on the space of large $k$-chains whose maximum have rank $i$.
	
	The reasoning of part \ref{chLarg} is the same with the weight: our aim is to find polynomial relations in $k$ between large $k$-chains, and then evaluate them at $k=-1$. Therefore, we will obtain: 
	
	\begin{equation}
	\Zmut = \sum_{n \geq 1} \sum_ {i \geq 0} \operatorname{WH_i}(\widehat{\operatorname{HT_n}}) t^i \frac{x^n}{n!}.
	\end{equation}
	
	\subsection{Connections between cycle indices}
	
	Relations between species of part \ref{rel esp} give the following relations when we take the weight into account: 
	
	\begin{prop}Series $\Zkt$, $\Zpkt$, $\Zakt$ and $\Zpakt$ satisfy the following relations:
	
	\begin{equation} \label{rel diss Zt}
	\Zkt+\Zpakt=\Zpkt+\Zakt,
	\end{equation}
	
	\begin{equation} \label{Zpkt}
	\Zpk = \frac{p_1}{t} \times (1 + \comm \circ \left( \frac{t\Zpkmut - p_1}{p_1}  \circ t \Zpkt \right),
	\end{equation}
	
	\begin{equation} \label{Zakt}
	\Zakt = (\Zkmut - \frac{p_1}{t} )\circ (t \Zpkt),
	\end{equation}
	
	\begin{equation} \label{Zpakt}
	\Zpakt = (\Zpkmut - \frac{p_1}{t}  )\circ (t \Zpkt),
	\end{equation}
	
	\begin{equation} \label{dert}
	p_1 \frac{\partial \Zkt }{\partial p_1} = \Zpkt.
	\end{equation}
	
	\end{prop}

	\subsection{New pointed chains}

We need two new kinds of pointed chains. Therefore, we will denote: 

\begin{itemize}
\item by $\mathcal{H}^A_{{k,t}}$, the species associated with large weighted hypertrees $k$-chains, whose maximum is an edge-pointed hypertree, and by $\ZAkt$  the associated cycle index.
\item by $\mathcal{H}^{pA}_{{k,t}}$, the species associated with large weighted hypertrees $k$-chains whose maximum is an edge-pointed rooted hypertree and $\ZPAkt$  the associated cycle index.
\end{itemize}

 Note that, by definition, the species $\mathcal{H}^a_{{1,t}}$ coincides with the species $\mathcal{H}^A_{{1,t}}$ and that the species $\mathcal{H}^{pa}_{{1,t}}$ coincides with species $\mathcal{H}^{pA}_{{1,t}}$ .
 
 The previous species are related with the other pointed hypertrees species by the following theorem: 
 
 \begin{thm} The species $\mathcal{H}^{A}_{{k,t}}$ and $\mathcal{H}^{pA}_{{k,t}}$ satisfy: 
  \begin{equation}
 \mathcal{H}^{A}_{{k,t}}=\mathcal{H}^A_{ { k-1,t } } \circ (t \mathcal{H}^p_{ { k,t }}),
 \end{equation}
  \begin{equation}
 \mathcal{H}^{pA}_{{k,t}}=\mathcal{H}^{pA}_{{k-1,t}} \circ (t \mathcal{H}^p_{{k,t}}),
 \end{equation}
 \begin{equation}
 \mathcal{H}_{{k,t}}+\mathcal{H}^{pA}_{{k,t}}=\mathcal{H}^{p}_{{k,t}}+\mathcal{H}^{A}_{{k,t}}.
 \end{equation}
 \end{thm}
 
 \begin{proof}
Pointing an edge in the maximum is the same as pointing an edge in the minimum
and pointing an edge in the set of distinguished edges thus obtained in the maximum of the chain. Using the proof of proposition \ref{eHak} and the previous statement give the first relation.
 
If we distinguish a vertex (root) in the chain, we obtain the second relation.

The third relation is obtained by the same reasoning as in paragraph \ref{paragraphe princ de diss} on the dissymmetry principle.
 
 \end{proof}
 
 	This implies the following relations: 
	\begin{cor} Series $\ZAkt$ and $\ZPAkt$ satisfy: 
	\begin{equation} \label{ZAkt}
 \ZAkt = \ZAkmut \circ (t \Zpkt),
 \end{equation}
  \begin{equation}\label{ZPAkt}
 \ZPAkt =\ZPAkmut \circ (t \Zpkt).
 \end{equation}
	\end{cor}

	\subsection{The $\HAL$ series} 	
	
	We recall here the definitions of $\HAL$ series introduced in \cite{ChHyp}.
 	\begin{defi} Series $\HAL$, $\HALp$, $\HALpa$ and $\HALa$ are the series defined by the following functional equations: 
 	\begin{equation}
 	\HALpa=p_1 \left(\frac{p_1}{1+tp_1} \circ \comm \circ (p_1 + (-t) \HALpa) \right),
 	\end{equation}
	\begin{equation}
	\HALp=p_1(\Stlie \circ \comm \circ (p_1 + (-t) \HALpa)),
	\end{equation} 	
 	\begin{equation} \label{5.16}
	\HALa=(\comm - p_1) \circ (p_1 + (-t) \HALpa),
	\end{equation}
	\begin{equation}
	\HAL=\HALp + \HALa - \HALpa.
	\end{equation} 	
 	\end{defi}	
 	
 	We introduce the series $\SWt$, defined by: 
		\begin{equation*}  
		(t \perm -t p_1 + p_1) \circ \SWt = \SWt \circ (t \perm -t p_1 + p_1)= p_1.
		\end{equation*}

	\begin{prop} Series $\SWt$ satisfies:
	\begin{equation} \label{CommSWt}
	\comm \circ \SWt = \frac{p_1 - \SWt}{t \SWt}.
	\end{equation}
	\end{prop} 
	
	\begin{proof}
	By definition, we have: 
	\begin{equation*}
	(\perm - p_1) \circ \SWt = \frac{p_1 - \SWt}{t}	.
	\end{equation*}
	
	However, $\perm$ satisfies: $\perm = p_1 (1+ \comm)$, 
	
	hence the result.
	\end{proof}	
 	
 	The following theorem gives explicit expressions for $\HAL$ series in terms of $\SWt$. 
 	
 	\begin{thm}\label{expr HAL} The series $\HAL$, $\HALp$, $\HALpa$ and $\HALa$ satisfy: 
 	\begin{equation}
 	\HALpa=\frac{p_1 - \SWt}{t},
 	\end{equation} 	
 	\begin{equation}
 	\HALa= (\comm - p_1) \circ \SWt,
 	\end{equation}
 	\begin{equation}
 	\HALp=\frac{p_1}{t} \left( \Slie \circ \frac{p_1 - \SWt}{\SWt} \right).
 	\end{equation}
 	where $\Slie$ is the series satisfying $\Slie \circ \comm = \comm \circ \Slie = p_1$.
 	\end{thm}
 	
 	\begin{proof}
 	\begin{enumerate}
 	\item Applying equation \eqref{CommSWt}, a computation gives: 
 	\begin{equation*}
 	p_1(\frac{\comm \circ \SWt}{1+t \comm \circ \SWt}) = p_1 \frac{p_1 - \SWt}{t \SWt + t p_1 - t \SWt},
 	\end{equation*}
 	hence the relation: 
 	\begin{equation*}
 	\frac{p_1 - \SWt}{t}=p_1(\frac{p_1}{1+tp_1}) \circ \comm \circ (p_1 + (-t) \frac{p_1 - \SWt}{t}).
 	\end{equation*}
 	The series $\HALpa$ and $\frac{p_1 - \SWt}{t}$ satisfy the same functional equation. Moreover if we know the first $n$ terms of a solution of this equation, the equation gives the $n+1$-th one: there is a unique solution of this equation, such that the coefficient of $x^0$ vanishes. Therefore, $\HALpa$ and $\frac{p_1 - \SWt}{t}$ are equals.
 	
 	: they are hence equals.
 	
 	\item The second equality results from the first one and equation \eqref{5.16} because the series $\SWt$ satisfies: 
 	\begin{equation*}
 	p_1 + (-t) \HALpa = \SWt.
 	\end{equation*}
 	
 	\item According to the first relation of the proposition, the series $\HALp$ satisfies: 
 	\begin{equation*}
 	\HALp=p_1(\Stlie \circ \comm \circ \SWt ).
 	\end{equation*}
 	
 	The equality $\Stlie = \frac{1}{t} \Slie \circ (t p_1)$ implies: 
 	\begin{equation*}
 	\HALp=\frac{p_1}{t}(\Slie \circ t\comm \circ \SWt ).
 	\end{equation*}
 	
 	Applying equation \eqref{CommSWt}, we get the result. 	
 	
	\end{enumerate}

 	\end{proof}

	\subsection{Character computation}
	
		\subsubsection{Computation of series for $k=0$}
		
		We can compute the following series: 
		
		\begin{prop}\begin{enumerate}
		\item The series $\Zot$ can be expressed as: 
		\begin{equation} \label{5.22}
		\Zot=\comm - p_1 +\frac{p_1}{t}.
		\end{equation}
		\item The series $\Zpot$ can be expressed as: 
		\begin{equation}\label{5.23}
		\Zpot=\perm-p_1+\frac{p_1}{t}=p_1\comm +\frac{p_1}{t}.
		\end{equation}
		The series $t \Zpot$ is then the inverse of series $\SWt$ for substitution.
		\item The series $\ZAot$ can be expressed as: 
		\begin{equation}\label{5.24}
		\ZAot=\comm - p_1.
		\end{equation}	
		\item The series $\ZPAot$ can be expressed as: 
		\begin{equation} \label{5.25}
		\ZPAot=\perm - p_1 = p_1 \comm.
		\end{equation}
		\end{enumerate}
		\end{prop}
	
	\begin{proof}
	\begin{enumerate} 
	\item The only hypertrees chain fixed by the action of an element $\sigma$ of the symmetric group $\mathfrak{S}_n$ is the empty chain. Nevertheless, the weight of the empty chain is $1$, except for $n=1$, where it is equal to $\frac{1}{t}$. Therefore the series $\Zot$ only differs from $\comm$ for $n=1$, hence the result.
	\item As $p_1 \frac{\partial \comm }{\partial p_1} = \perm$, the result comes from relation \eqref{dert} with $k=0$.
	
	\item By definition, $\ZAut = \Zaut$, with relations \eqref{ZAkt} and \eqref{Zakt}, the series $\ZAot$ satisfies: 
	\begin{equation*}
	\ZAot=\Zot - \frac{p_1}{t} = \comm - p_1.
	\end{equation*}
		
	\item By definition, $\ZPAut = \Zpaut$, with relations \eqref{ZPAkt} and \eqref{Zpakt}, the series $\ZPAot$ satisfies: 
	\begin{equation*}
	\ZPAot=\Zpot - \frac{p_1}{t} = \perm - p_1 = p_1 \comm.
	\end{equation*}
	
	\end{enumerate}
	\end{proof}
	
				\subsubsection{Computation of the series for $k=-1$}
				
		The following theorem refines the computation of the characteristic polynomial in \cite{ChHyp}, proves the conjecture of \cite[Conjecture 5.3]{ChHyp} and links the action of the symmetric group on Whitney homology of the hypertree poset with the action of symmetric group on a set of hypertrees decorated by the $\operatorname{Lie}$ operad.
				
		\begin{thm} \label{goalt}
		\begin{enumerate}
		\item The series $\ZPAmut$ satisfies: 
		\begin{equation}
		\ZPAmut = \frac{p_1 - \SWt}{t} =\HALpa.
		\end{equation}
		\item The series $\ZAmut$ satisfies:
		\begin{equation}
		\ZAmut= (\comm - p_1) \circ \SWt =\HALa.
		\end{equation}	
		\item The series $\Zpmut$ satisfies:
		\begin{equation}
		\Zpmut= \frac{p_1}{t}(1+\Slie \circ \frac{p_1 - \SWt}{\SWt})= \HALp + \frac{p_1}{t}.
		\end{equation}
		\item The series $\Zmut$ satisfies:
		\begin{equation}
		\Zmut=\HAL+\frac{p_1}{t}.
		\end{equation}
		\end{enumerate}
		\end{thm}
		
		\begin{proof}
		The right part of equalities is given by theorem \ref{expr HAL}.
		
		\begin{enumerate}
		\item Relation \eqref{ZPAkt} with $k=0$ gives, together with equations \eqref{5.23} and \eqref{5.25}: 
		\begin{equation*}
		\ZPAmut = (p_1 \comm) \circ \SWt.
		\end{equation*}
		
		We then conclude thanks to equation \eqref{CommSWt}.
		
		\item Relation \eqref{ZPAkt} with $k=0$ gives, together with equations \eqref{5.23} and \eqref{5.24}: 
		 
		\begin{equation*}
		\ZAmut = (\comm - p_1) \circ \SWt,
		\end{equation*}
		hence the result.
		
		\item As $\SWt$ is the inverse of $t \Zpot$, relation \eqref{Zpkt} with $k=0$ gives: 
		\begin{equation*}
		p_1 = \SWt (1+ \comm \circ \frac{t \Zpmut - p_1}{p_1})
		\end{equation*}
		
		as $\Slie \circ \comm = p_1$ according to \cite{ChHyp}, we obtain: 
		\begin{equation*}
		\Slie \circ \frac{p_1 - \SWt}{\SWt}=  \frac{t \Zpmut - p_1}{p_1}.
		\end{equation*}
		
		We thus obtain the result.
		
		\item This relation comes from previous relations associated with the dissymmetry principle.
		\end{enumerate}
		\end{proof}

\appendix

\section{Reminder on species} 
\label{rappel espèces}

We give in this part only a brief reminder on species. The reader will find more on this subject in \cite{BLL}.

\begin{defi} A \emph{species} $\operatorname{F}$ is a functor from the category of finite sets and bijections to the category of finite sets. To a finite set $I$, the species $\operatorname{F}$ associate a finite set $\operatorname{F}(I)$ independent from the nature of $I$.
\end{defi}

\begin{exple} 
\begin{itemize}
\item The map which associates to a finite set $I$ the set of total orders on $I$ is a species, called the linear order species and denoted by $L$.
\item The map which associates to a finite set $I$ the set  $\{I\}$ is a species, called the set species and denoted by $E$.
\item The map defined for all finite set $I$ by: 
\begin{equation*}
I \mapsto \left\lbrace
\begin{array}{ll}
\{I\} & \text{if } \# I = 1, \\
\emptyset & \text{otherwise},
\end{array} \right.
\end{equation*}
is a species, called singleton species and denoted by $X$.
\item The map defined for all finite set $I$ by: 
\begin{equation*}
I \mapsto \left\lbrace
\begin{array}{ll}
\{I\} & \text{if } \# I \geq 1, \\
\emptyset & \text{otherwise},
\end{array} \right.
\end{equation*}
is a species denoted by ${Comm}$, and called species associated with the $\operatorname{Comm}$ operad.
\item The map which associates to a finite set $I$ the set $I$ is a species, called the pointed set species and denoted by ${Perm}$. It is associated with the $\operatorname{Perm}$ operad.
\item The map which associates to a finite set $I$ the set of labelled rooted trees with labels in $I$ is a species denoted by $PreLie$, associated with the $\operatorname{PreLie}$ operad.
\end{itemize}
\end{exple}

To each species $F$, we can associate the following generating series: 
\begin{equation*}
C_F(x)=\sum_{n \geq 0} \# F(\{1,\ldots, n\}) \frac{x^n}{n!}. 
\end{equation*}

\begin{exple} The generating series of species defined previously are: 
\begin{itemize}
\item $C_L(x)=\frac{1}{1-x}$,
\item $C_E(x)=\exp(x)$,
\item $C_X(x)=x$,
\item $C_{\operatorname{Comm}}(x)=\exp(x)-1$ .
\end{itemize}
\end{exple}

The following operations can be defined on species:
 
\begin{defi} Let $F$ and $G$ be two species. We define the following operations on species: 
\begin{itemize}
\item $F'(I)=F(I \sqcup \{\bullet \} )$, (differentiation)
\item $(F + G )(I)=F(I) \sqcup G(I)$, (addition)
\item $(F\times G )(I)=F(I) \times G(I)$, (product)
\item $(F \circ G) (I)=\bigsqcup_{\pi \in \mathcal{P}(I)} F(\pi) \times \prod_{J \in \pi} G(J) $, (substitution)

where $\mathcal{P}(I)$ runs on the set of partitions of $I$.
\end{itemize}
\end{defi}

We have the following property: 
\begin{prop} Let $F$ and $G$ be two species. Their generating series satisfy: 
\begin{itemize}
\item $C_{F'}=C_F'$,
\item $C_{F + G}=C_F + C_G$,
\item $C_{F \times G}=C_F \times C_G$,
\item $C_{F \circ G}=C_F \circ C_G$.
\end{itemize}
\end{prop}

\section{Reminder on cycle index}	

\label{Annexe ind cycl}
	
	Let ${F}$ be a species. We can associate a formal power series to it: its cycle index. The reader can consult \cite{BLL} for a reference on this subject. This formal power series is a symmetric function defined as follow:
	
	\begin{defi}
	The \textbf{cycle index} of a species $F$ is the formal power series in an infinite number of variables $\p=(p_1, p_2, p_3, \ldots )$ defined by: 
	
	$$\textbf{C}_F(\p)= \sum_{n \geq 0} \frac{1}{n!} \left( \sum_{\sigma \in \mathfrak{S}_n} F^\sigma p_1^{\sigma_1} p_2^{\sigma_2} p_3^{\sigma_3} \ldots \right), $$	
	
	where $F^\sigma$ stands for the set of $F$-structures fixed under the action of $\sigma$ and where $\sigma_i$ is the number of cycles of length $i$ in the decomposition of $\sigma$ into disjoint cycles.
		\end{defi}
		
		We can define the following operations on cycle indices.
		
		\begin{defi}
		The operations $+$ and $\times$ on cycle indices are the same as on formal series.

		For $f=f(\p)$ and $g=g(\p)$, plethystic substitution $f \circ g$ is defined by: 
		\begin{equation*}
		f \circ g (\p) = f(g(p_1, p_2, p_3, \ldots), g(p_2, p_4, p_6, \ldots),  \ldots, g(p_{k}, p_{2k}, p_{3k}, \ldots), \ldots).
		\end{equation*}
		It is left-linear.
		\end{defi}

		This operations satisfy: 

		\begin{prop}
		Let $F$ and $G$ be two species. Their cycle indices satisfy: 
		\begin{equation*}
		\begin{array}{rlrl}
		\textbf{C}_{F+G}&=\textbf{C}_F+\textbf{C}_G, & \textbf{C}_{F \times G}&=\textbf{C}_F \times \textbf{C}_G, \\
		\textbf{C}_{F\circ G}&=\textbf{C}_F \circ \textbf{C}_G, & \textbf{C}_{F'}&= \frac{\partial \textbf{C}_{F} }{\partial p_1}.
		\end{array}
		\end{equation*}
		\end{prop}
				
		Moreover, we define the following operation: 
		
		\begin{defi}  The \emph{suspension} $\Sigma_t$ of a cycle index $f(p_1, p_2, p_3, \ldots)$ is defined by: 
		\begin{equation*}
		\Sigma_t f = - \frac{1}{t} f(-tp_1, -t^2 p_2, -t^3 p_3, \ldots). 
		\end{equation*}
		By convention, we will write $\Sigma$ for the suspension in $t=1$.
		\end{defi}

\bibliographystyle{alpha}
\bibliography{bibli}

\end{document}